\documentclass[12pt]{article}
\usepackage[T1]{fontenc}
\usepackage[utf8]{inputenc}
\usepackage{authblk}
\usepackage[misc]{ifsym}
\usepackage{dsfont}
\usepackage{mathrsfs}
\usepackage{amsmath,amssymb}
\usepackage{amsfonts}

\usepackage{amsthm}
\usepackage{graphicx}
\usepackage{subfigure}
\usepackage{xcolor}

\usepackage{bbm}
\usepackage{mathrsfs}
\usepackage{txfonts}
\usepackage{color}
\usepackage{pict2e}
\usepackage{float}
\usepackage{palatino,epsfig,latexsym}
\usepackage{epsf,xy,epic,amscd}
\usepackage{lineno}

\usepackage{bibentry}
\usepackage[colorlinks,linkcolor=blue,anchorcolor=blue,citecolor=blue,]{hyperref}

\theoremstyle{plain}
\newtheorem{theorem}{Theorem}[section]
\newtheorem{lemma}[theorem]{Lemma}

\newtheorem{proposition}[theorem]{Proposition}
\newtheorem{corollary}[theorem]{Corollary}

\newtheorem{Bounded Diameter Lemma}[theorem]{Bounded Diameter Lemma}
\theoremstyle{definition}
\newtheorem{definition}[theorem]{Definition}

\newtheorem{remark}[theorem]{Remark}

\newcommand{\Hmm}[1]{\leavevmode{\marginpar{\tiny%
$\hbox to 0mm{\hspace*{-0.5mm}$\leftarrow$\hss}%
\vcenter{\vrule depth 0.1mm height 0.1mm width \the\marginparwidth}%
\hbox to
0mm{\hss$\rightarrow$\hspace*{-0.5mm}}$\\\relax\raggedright #1}}}

\hfuzz=\maxdimen
\DeclareFixedFont{\Acknowledgment}{OT1}{cmr}{bx}{n}{14pt}
\begin{document}
\numberwithin{equation}{section}

\title{On  the regularity 
	 of Dirichlet problem for fully non-linear elliptic equations on Hermitian manifolds
 }
 \author{Rirong Yuan\thanks{School of Mathematics, \,South China University of Technology, \,Guangzhou  510641,   \,China \\ \indent \indent Email address:\,yuanrr@scut.edu.cn
 }
}

\date{}

\maketitle

\begin{abstract}
We derive the solvability and regularity of the Dirichlet problem for fully non-linear elliptic equations possibly with degenerate right-hand side on Hermitian manifolds,  through establishing a quantitative version of boundary estimate under a subsolution assumption. 
 In addition, we construct the subsolution
when the background manifold is a product of a closed Hermitian manifold with a compact  
Riemann surface with boundary.

 

\end{abstract}

\tableofcontents

\section{Introduction}

Let $(M,J,\omega)$ be a compact Hermitian manifold of complex dimension $n\geq 2$ with boundary $\partial M$, where 
$\omega=\sqrt{-1} g_{i \bar{j}} d z^{i} \wedge d \bar z^{j}$
 denotes the K\"{a}hler form compatible with the complex structure $J$. Suppose $\chi=\sqrt{-1} \chi_{i \bar{j}} d z^{i} \wedge d \bar{z}^{j}$ is a smooth real $(1,1)$-form on $\bar{M}:=M\cup\partial M$. Given a $C^2$-function $u$ on $\bar{M}$, one can obtain a new real $(1,1)$-form $\mathfrak{g}[u]:=\chi+\sqrt{-1}\partial \overline{\partial} u.$

This paper is devoted to investigating the following Dirichlet problem
\begin{equation}\label{mainequ}
	\begin{aligned}
F(\mathfrak{g}[u])=\psi \text { in } M, \quad u=\varphi\, \text { on }\, \partial M.
\end{aligned}
\end{equation}
We assume
  the operator $F(\mathfrak{g}[u])$ takes the form
$$
F(\mathfrak{g}[u])=f(\lambda(\mathfrak{g}[u])),
$$
where $\lambda(\mathfrak{g}[u])=(\lambda_1,\cdots,\lambda_n)$ denote the eigenvalues of $\mathfrak{g}[u]$ with respect to $\omega$, and $f$ is a smooth symmetric function defined in an open symmetric convex cone $\Gamma\subset \mathbb R^n$ containing the positive cone 
\[\Gamma_n:=\left\{\lambda=(\lambda_1,\cdots,\lambda_n)\in \mathbb{R}^n: \mbox{ each } \lambda_i>0\right\}\subseteq\Gamma\]
with vertex at the origin and with the boundary $\partial \Gamma \neq \emptyset$.
 In addition, we have the following hypotheses on $f$:
\begin{equation}\label{elliptic}
f_{i}(\lambda):=\frac{\partial f}{\partial \lambda_{i}}(\lambda)>0 \,\text { in }\, \Gamma, \mbox{ } \forall 1 \leqslant i \leqslant n,
\end{equation}
\begin{equation}\label{concave}
	\begin{aligned}
	f \text { is concave in } \Gamma,	
	\end{aligned}
\end{equation}
\begin{equation}\label{addistruc}
\text { For any } \lambda \in \Gamma,\mbox{ } \lim _{t \rightarrow+\infty} f(t \lambda)>-\infty,
\end{equation}
\begin{equation}
	\label{unbounded}
	\begin{aligned}
\lim _{t \rightarrow+\infty} f\left(\lambda_{1}, \cdots, \lambda_{n-1}, \lambda_{n}+t\right)=\sup_{\Gamma} f, \mbox{ } \forall \lambda=\left(\lambda_{1}, \cdots, \lambda_{n}\right) \in \Gamma.
\end{aligned}
\end{equation}
In real variables, equations  of this type 
were  investigated
by Caffarelli-Nirenberg-Spruck \cite{CNS3},
which extended the work of Ivochkina \cite{Ivochkina1981}
on equations of Monge-Amp\`ere type.

Notice that equation~\eqref{mainequ} includes many equations as special cases. For instance, if we set $f(\lambda)=\sum_{i=1}^n\log\lambda_i$, it  
reads
  the complex Monge-Amp\`ere equation, which has played a significant role in Yau's proof of Calabi's conjecture \cite{Yau78}. In fact, over the past decades, there were many researches on complex Monge-Amp\`ere equation.
   Below we list parts of these works. 
  For $\chi = 0$ and $M=\Omega \subset \mathbb{C}^{n}$ a bounded strictly pseudoconvex domain, the Dirichlet problem for complex Monge-Ampère equation was solved by Caffarelli-Kohn-Nirenberg-Spruck \cite{CKNS2} in the class of plurisubharmonic functions.
   Caffarelli-Kohn-Nirenberg-Spruck's work was extended by Guan \cite{Guan1998The} to general bounded domains,  replacing 
   strictly pseudoconvex restriction to boundary by a subsolution assumption. Under the subsolution assumption, Guan-Li \cite{Guan2010Li} solved the Dirichlet problem for complex Monge-Amp\`ere equation on compact Hermitian manifolds with boundary.
     For $\psi\geq 0$, the Dirichlet problem for complex Monge-Amp\`ere equation becomes degenerate, which is much more complicated.  In \cite{Chen}, Chen solved the Dirichlet problem for homogeneous complex Monge-Amp\`ere equation on $M=X\times A$ and proved the existence of $C^{1,\alpha}$-regularity   (weak) geodesics in the space of K\"ahler potentials \cite{Donaldson99,Mabuchi87,Semmes92},  where $A=\mathbb{S}^1\times [0,1]$ and $X$ is a closed K\"ahler manifold. 
      The Dirichlet problem for degenerate complex Monge-Amp\`ere equation was further complemented by B{\l}ocki \cite{Blocki09geodesic}, Phong-Sturm \cite{Phong-Sturm2010}  and Boucksom \cite{Boucksom2012}.
          We refer to \cite{Chen2008Tian,Donaldson2002Holomorphic}  for complement and progress on understanding how geodesics and homogeneous complex Monge-Amp\`ere equation are related to the geometry of $X$.
    There are also several results on the Dirichlet problem~\eqref{mainequ} for   other equations. For $\chi=0$ and $M=\Omega$ a bounded domain in $\mathbb{C}^n$, 
       the Dirichlet problem~\eqref{mainequ}  was studied by Li \cite{LiSY2004}. 
    For complex inverse $\sigma_k$ equations on Hermitian manifolds,
    the Dirichlet problem was studied by Guan-Sun \cite{Guan2015Sun}.

The above results mainly 
concentrate on special cases of~\eqref{mainequ} and the proof more or less relies on specific structures of equations 
  or underlying manifolds, 
which  
 seems not adaptive
to more general cases. On the other hand, motivated by 
 increasing interests from 
 complex geometry and analysis, we are interested in investigating the Dirichlet problem~\eqref{mainequ} on curved complex manifolds, 
especially with degenerate right-hand side. 
Unfortunately, except the 
 above-mentioned works regarding to complex Monge-Amp\`ere equation \cite{Chen,Boucksom2012,Guan2010Li,Phong-Sturm2010} and complex inverse $\sigma_k$ equation \cite{Guan2015Sun}, 
 few progress has been made on 
 the 
  Dirichlet problem 
   for general equations on curved Hermitian manifolds. 
    The primary problem left open is to derive the gradient estimate as described below. To this end, in this paper we 
 set up a quantitative version of boundary estimate of the form
\begin{equation}
 \label{bdy-sec-estimate-quar1}
 \begin{aligned}
\sup _{\partial M} \Delta u \leqslant C\left(1+\sup _{M}|\nabla u|^{2}\right) 
\end{aligned}
\end{equation}
 and then solve the Dirichlet problem. 

First let us introduce some basic notions. 
We say a function $w\in C^{2}(\bar{M})$ is \emph{admissible} if
\[
\lambda(\mathfrak{g}[w]) \in \Gamma \text { in } \bar{M}.
\]
For an admissible function $\underline{u} \in C^{2}(\bar{M})$, it is called an  \emph{admissible subsolution} of the Dirichlet problem~\eqref{mainequ}, if
\begin{equation}
	\label{existenceofsubsolution}
	\begin{aligned} 
		f(\lambda(\mathfrak{g}[\underline{u}])) \geqslant \psi \text { in } \bar{M}, \quad \underline{u}=\varphi \text { on } \partial M.
	\end{aligned}
\end{equation}
Moreover, $\underline{u}$ is called a 
 \textit{strictly admissible subsolution}, if
\begin{equation}
	\label{existenceofsubsolution-2}
	\begin{aligned} 
		f(\lambda(\mathfrak{g}[\underline{u}]))>\psi \text { in } \bar{M}, \quad \underline{u}=\varphi \text { on } \partial M.
	\end{aligned}
\end{equation}
Meanwhile, we say the Dirichlet problem~\eqref{mainequ} is  non-degenerate if
\begin{equation}\label{nondegenerate}
\inf _{M} \psi>\sup _{\partial \Gamma} f.
\end{equation}
We say it is degenerate if
\begin{equation}\label{degenerate-RHS}
	\begin{aligned}
		\inf _{M} \psi=\sup _{\partial \Gamma} f
	\end{aligned}
\end{equation}
and
 \begin{equation}
	\label{continuity1}
	\begin{aligned}
		f\in C^\infty(\Gamma)\cap C(\overline{\Gamma}).
	\end{aligned}
\end{equation} 
Here
\[ 
\sup _{\partial \Gamma} f=\sup _{\lambda_{0} \in \partial \Gamma} \limsup _{\lambda \rightarrow \lambda_{0}} f(\lambda), \quad 	\overline{\Gamma} =\Gamma\cup\partial\Gamma.
\]
Throughout this paper, we always assume the boundary data $\varphi$ can be extended to a $C^{2,1}$-admissible function on $\bar{M}$. (Such an assumption is necessary for the solvability). For simplicity, we still denote it by $\varphi$. 

\begin{remark}
The perspective of subsolution was imposed by 
\cite{Guan1998The,Guan1993Boundary,Hoffman1992Boundary} as a vital tool to deal with the second order boundary estimates for Dirichlet problem of Monge-Amp\`ere equation on bounded domains. The concept of subsolutions has a great advantage in applications to some geometric problem 
 that it relaxes restrictions to the shape of boundary, see e.g. \cite{Chen,Guan1993Boundary,GuanP2002The,Guan2009Zhang}.
\end{remark}

Our main results are as follows.

\begin{theorem}	\label{thm1-1}
Let $(M,J,\omega)$ be a compact Hermitian manifold with smooth boundary. Let $f$ 
satisfy the hypotheses \eqref{elliptic}-\eqref{unbounded}. Assume the 
data $\varphi \in C^{k+2, \alpha}(\partial M)$ and $\psi \in C^{k, \alpha}(\bar{M})$ with $k \geqslant 2$, $0<\alpha<1$ satisfy~\eqref{nondegenerate} and  support  a $C^{2,1}$-admissible subsolution. Then the Dirichlet problem~\eqref{mainequ} possesses a unique admissible solution $u \in C^{k+2, \alpha}(\bar{M})$.
\end{theorem}

Let $\Gamma_\infty$ be the projection of $\Gamma$ to the subspace of former $n-1$ subscripts. Namely, $\left(\lambda_{1}, \cdots, \lambda_{n-1}\right)\in \Gamma_\infty$ if and only if there exists $\lambda_n>0$ such that $\left(\lambda_{1}, \cdots, \lambda_{n-1}, \lambda_{n}\right) \in \Gamma$.

\begin{theorem}
	\label{thm1-2}
In Theorem~\ref{thm1-1} the hypothesis~\eqref{unbounded} on $f$ can be dropped if
the Levi form $L_{\partial M}$ of $\partial M$ satisfies
\begin{equation}\label{bdry-assumption1}
\left(-\kappa_{1}, \cdots,-\kappa_{n-1}\right) \in \overline{\Gamma}_{\infty} \text { on } \partial M
\end{equation}
where  $\kappa_{1}, \cdots, \kappa_{n-1}$ denote the eigenvalues of Levi form $L_{\partial M}$ of $\partial M$ with respect to $\omega^{\prime}=\left.\omega\right|_{T_{\partial M} \cap J T_{\partial M}}$, and $\overline{\Gamma}_{\infty}$ is the closure of $\Gamma_{\infty}$. 
\end{theorem}

To prove the above results, 
 it remains to prove gradient estimate.
  However,
it is still 
 highly mysterious to prove gradient bound directly for \eqref{mainequ} on complex manifolds, as B{\l}ocki \cite{Blocki09gradient}, 	Hanani \cite{Hanani1996}  and Guan-Li \cite{Guan2010Li}
 did for complex Monge-Amp\`ere equation, to be compared with
 the gradient estimate 
established by Li \cite{LiYY1990} and Urbas \cite{Urbas2002} for
fully nonlinear elliptic equations on Riemannian manifolds.  
 The complex version of Li's result was obtained by the author \cite{yuan2021cvpde} on K\"ahler manifolds with nonnegative orthogonal bisectional curvature; while the trick used in  
\cite{Urbas2002} 
does not work in complex variables any more. 
The trouble comes from
 the lack of understanding of pure complex derivatives $u_{ij}$, $u_{\bar i\bar j}$ when restricted to 
mixed complex derivatives 
$u_{i\bar j}$.  
Blow-up argument is an alternative approach to derive the gradient estimate. 
When $(M,J,\omega)$ is a closed Hermitian manifold,  
Sz\'ekelyhidi 
\cite{Gabor}  proved the second estimate
for equation \eqref{mainequ} 
\begin{equation}
 	\label{sec-estimate-quar1}
	\begin{aligned}
		\sup_{M}\Delta u\leq C \left(1+\sup_{M} |\nabla u|^2\right) 
	\end{aligned}
\end{equation}
and then used it to derive gradient estimate via a blow-up argument. 
Such a blow-up argument using \eqref{sec-estimate-quar1} 
 appeared in literature that has been done by
Chen \cite{Chen}, complemented by \cite{Boucksom2012,Phong-Sturm2010},
for Dirichlet problem of complex 
Monge-Amp\`ere equation, 
 and by Dinew-Ko{\l}odziej \cite{Dinew2017Kolo} for complex $k$-Hessian equations on
closed K\"ahler manifolds using Hou-Ma-Wu's 
  second estimate \cite[Theorem 1.1]{HouMaWu2010}.
 We also refer respectively to \cite{STW17,TW17,TW19} and \cite{ZhangDk} for related works 
devoting to 
Gauduchon's conjecture and   complex $k$-Hessian equations on closed Hermitian manifolds.

To show 
Theorems \ref{thm1-1} and \ref{thm1-2}, 
 a specific problem that we have in mind is to establish 
the boundary estimate \eqref{bdy-sec-estimate-quar1}. Unfortunately, the proof in  \cite{Chen,Blocki09geodesic,Phong-Sturm2010}   relies heavily upon the specific structure of Monge-Amp\`ere operator, 
 which cannot be adapted to general equations.
In 
this paper we set up
 Lemmas 
 \ref{lemma3.4} and \ref{yuan's-quantitative-lemma} 
 in an attempt to bound 
 the double normal derivative of solutions on the boundary  in a quantitative form. Subsequently, we achieve the goal as follows.
\begin{theorem}
\label{thm1-bdy}

Let $(M,J,\omega)$ and $f$ be as in Theorem~\ref{thm1-1}. Assume that $\varphi \in C^{3}(\partial M)$ and $\psi \in C^{1}(\bar{M})$ satisfy~\eqref{nondegenerate} and support a $C^{2}$-admissible subsolution $\underline{u}$. Then any admissible solution $u \in C^{3}(M) \cap C^{2}(\bar{M})$ to~\eqref{mainequ} satisfies
\[
\sup _{\partial M} \Delta u \leqslant C\left(1+\sup _{M}|\nabla u|^{2}\right),
\]
where $C$ is a uniform positive constant depending on $|\varphi|_{C^{3}(\bar{M})}$, $|\nabla u|_{C^{0}(\partial M)}$, $|\underline{u}|_{C^{2}(\bar{M})}$, $|\psi|_{C^{1}(M)}$, $\partial M$ up to third derivatives, and other known data under control (but not on $\left.\sup _{M}|\nabla u|\right)$.

Moreover, the hypothesis~\eqref{unbounded} can be dropped and the constant $C$ is independent of $(\delta_{\psi, f})^{-1}$, provided $\partial M$ satisfies~\eqref{bdry-assumption1}, where
\[
\delta_{\psi,f}=\inf_{M}\psi-\sup_{\partial \Gamma} f.
\]
\end{theorem}

\begin{remark}
	In the proof of estimates, it only requires the subsolution $\underline{u}\in C^2(\bar M)$.
\end{remark}

Note that the above estimate is fairly delicate. It does not depend on $(\delta_{\psi, f})^{-1}$ under assumption~\eqref{bdry-assumption1}. 
As a result,
together with Lemma \ref{lemma1-con-addi},
  we can solve the Dirichlet problem for 
degenerate equations.  

\begin{theorem} \label{thm2-diri-de}
Let $(M,J,\omega)$ be a compact Hermitian manifold with smooth boundary subject to~\eqref{bdry-assumption1}, and let $f$ satisfy~\eqref{elliptic},~\eqref{concave}  and~\eqref{continuity1}.
Assume  $\varphi \in C^{2,1}(\partial M)$ and $\psi \in C^{1,1}(\bar{M})$ satisfy~\eqref{degenerate-RHS} and support a strictly admissible subsolution $\underline{u}\in C^{2,1}(\bar{M})$. Then there exists a (weak) solution  $u \in C^{1, \alpha}(\bar{M})$ to the Dirichlet problem~\eqref{mainequ} with $\forall 0<\alpha<1$ such that
\[
\lambda(\mathfrak{g}[u]) \in \overline{\Gamma} \mbox{ in }  \bar{M}, \quad    \Delta u \in L^{\infty}(\bar{M}).
\]
\end{theorem}

\begin{remark}
	To some sense, this may be
	the first breakthrough for Dirichlet problem for general
	 degenerate fully nonlinear elliptic equations of the type  \eqref{mainequ} on complex manifolds. 
\end{remark}

\begin{remark}
	When $\Gamma=\Gamma_n$, 
	Theorem \ref{thm1-1} gives back a result of the  author~\cite[Theorem 1]{yuan2021cvpde} with a different method.
	The case   $(f,\Gamma)=(\sigma_k^{1/k},\Gamma_k)$ was also proved by Collins-Picard  \cite{Collins2019Picard} independently; while the right-hand side of equation considered there does not include degenerate case, to be compared with Theorem \ref{thm2-diri-de}.
\end{remark}




It is a remarkable fact that the Dirichlet problem is not always solvable without the subsolution assumption. 
A natural problem is to construct the subsolutions.
 Unfortunately, except on 
 certain domains in Euclidean spaces \cite{CKNS2,CNS3,LiSY2004}, few progress has been made on general manifolds. 

We confirm the subsolution assumption 
when the manifold is a product.
Without specific clarification, $(X, J_X, \omega_X)$ is a closed 
Hermitian manifold of complex dimension $n-1$, and
  $(S, J_S, \omega_S)$ is a compact Riemann surface with boundary $\partial S$. Let $\pi_1: X\times S\rightarrow X$
and
$\pi_2: X\times S\rightarrow S$
denote the natural projections.
On the product $(M,J,\omega)=(X\times S,J,\omega)$,
  we are able to construct  strictly  {admissible} subsolutions  
for the Dirichlet problem, provided 
\begin{equation}
	\label{cone-condition1}
	\begin{aligned}
		\lim_{t\rightarrow +\infty} f\left(\lambda(\mathfrak{g}[\varphi]+ t\pi_2^*\omega_S)\right)>\psi  
		 \mbox{ }  \mbox{ in } \bar M.
	\end{aligned}
\end{equation}
Here  $J$ is the induced complex structure and  $\omega$ is a Hermitian metric compatible with $J$ (but not necessary the standard product metric $\omega=\pi_1^*\omega_X+\pi_2^*\omega_S$).
It is noteworthy that  
\eqref{cone-condition1} automatically holds when $f$ satisfies \eqref{unbounded}. 

 
As consequences, we deduce the following results.
\begin{theorem}
	\label{mainthm-10-degenerate-2}  
	Let $(M, J,\omega)=(X\times S,J, \omega)$ be as above with $\partial S\in C^{\infty}$.  
	Suppose  in addition that 	
	 \eqref{elliptic}-\eqref{addistruc} hold.
	Then for
		$\varphi\in C^{\infty}(\partial M)$, 
		$\psi \in C^{\infty}(\bar M)$ satisfying \eqref{nondegenerate} and \eqref{cone-condition1},
		the Dirichlet problem \eqref{mainequ} 
		admits a unique smooth admissible solution. 
		
\end{theorem}

\begin{theorem}
	\label{mainthm-10-degenerate}  
	Let $(M, J,\omega)=(X\times S,J, \omega)$ be as above 
	with $\partial S\in C^{2,1}$, and let
 $f$ satisfy \eqref{elliptic}, \eqref{concave} and \eqref{continuity1}.  
	Given $\varphi\in C^{2,1}(\partial M)$ and $\psi\in C^{1,1}(\bar M)$ satisfying \eqref{degenerate-RHS} and \eqref{cone-condition1}, the Dirichlet problem \eqref{mainequ} 
		has a weak solution $u$ with   
		\begin{equation}
			\begin{aligned}
				u\in C^{1,\alpha}(\bar M), 
				\mbox{  } \forall 0<\alpha<1, \mbox{ } \Delta u \in L^{\infty}(\bar M),
				\mbox{  }
				\lambda(\mathfrak{g}[u])\in \bar \Gamma \mbox{ in } \bar M. \nonumber
			\end{aligned}
		\end{equation}

\end{theorem}




\begin{remark}
	The degenerate fully nonlinear elliptic equations on the product $X\times S$ have many applications in geometry.
When $S=\mathbb{S}^1\times [0,1]$ and $f=\sigma_n^{1/n}$, 
Theorem \ref{mainthm-10-degenerate} immediately yields
 Chen's 
  existence and regularity of (weak) geodesics in the space of K\"ahler potentials $\mathcal{H}_{\omega_X}$. 
 Moreover, as shown by Donaldson \cite{Donaldson99}, for any  compact Riemann surface $S$ with boundary, the homogeneous complex Monge-Amp\`ere equation on $X\times S$ is 
 closely related to
 the Wess-Zumino-Witten equation for a map
 from $S$ to $\mathcal{H}_{\omega_X}$. Our results 
  regarding to 
 the regularity of weak solutions and  the construction of subsolutions apply to 
the Wess-Zumino-Witten equation and possibly enable one to attack some related problems. 
\end{remark}


In conclusion, in this paper we first derive the quantitative boundary estimate and then  
  solve the Dirichlet problem for fully nonlinear elliptic equations, possibly with degenerate right-hand side.
No matter the equations are degenerate or not,  
the 
existence results for Dirichlet problem of general fully nonlinear elliptic equations are rarely known until this work.  
In addition, we construct subsolutions on products, which are of numerous interests. Our results   
 include the existence and regularity results  of  Chen  \cite{Chen} as a special case and   extend extensively Sz\'ekelyhidi's work \cite{Gabor} to complex manifolds with boundary. 


 \vspace{1mm}
The paper is organized as follows. In Section \ref{Preliminaries} we sketch the proof of main estimates.
 In Section \ref{sec3} we propose two lemmas which are key ingredients to understand the quantitative version of boundary estimate for double normal derivative. In Sections \ref{sec4}
and \ref{sec5} we derive the quantitative boundary estimate for double normal and tangential-normal derivatives, respectively. The quantitative boundary estimate is then derived as a consequence.
In Section \ref{solvingequation}
we complete the proof of Theorems \ref{thm1-1}, \ref{thm1-2} and \ref{thm2-diri-de}.
In Section \ref{sec6} the subsolutions are constructed when 
 $M$ is a product of a closed complex manifold 
 with a compact Riemann surface with boundary. 
In addition, we study Dirichlet problem on such products with less regularity assumptions on boundary. 
The uniqueness of weak solutions and construction of subsolutions on more general products are 
discussed in Section \ref{sec9}.


 \subsection*{Acknowledgements}
 Part of this work
 was done while the author  was visiting University of Science and Technology of China in semester  2016/2017. 
 The author wishes to express his gratitude to  Professor Xi Zhang for his encouragement and  
 to the University
 for the  hospitality. 
The author is supported by the National Natural Science Foundation of China under grant 11801587.

 \section{Sketch of the proof of main estimates}
 \label{Preliminaries}

We first summarize the notations as follows.
 \begin{itemize}
 
 	\item For $\sigma$, we denote 
 	\begin{equation}
 		\begin{aligned}
 		\label{def-levelset}
 		\Gamma^\sigma=\{\lambda\in\Gamma: f(\lambda)>\sigma\}, \quad
 		\partial\Gamma^\sigma=\{\lambda\in\Gamma: f(\lambda)=\sigma\}.
 		\end{aligned}
 	\end{equation} 
 	\item $\vec{\bf 1}=(1,\cdots,1)\in \mathbb{R}^n.$
 	
 		\item For the solution $u$ and subsolution $\underline{u}$, we denote $\mathfrak{g}=\mathfrak{g}[u]$ and 
 	$ \underline{\mathfrak{g}}={\mathfrak{g}}[\underline{u}]$, respectively. 
 	
 \end{itemize}

Condition \eqref{elliptic} ensures \eqref{mainequ} to be elliptic at admissible functions,
while  \eqref{concave} implies that the operator $F(A)=f(\lambda(A))$ is concave with respect to $A$ subject to $\lambda(A)\in\Gamma$. 
Consequently, according to Evans-Krylov theorem  \cite{Evans82,Krylov83},
adapted to complex setting (see e.g. \cite{TWWYEvansKrylov2015}), together with 
 Schauder theory, higher order estimates for admissible solutions can be derived from 
uniform bound of complex hessian 
\begin{equation}
	\label{estimate00}
	|\partial\overline{\partial} u|\leq C.
\end{equation}
The existence results then follow from standard continuity method. 

\subsection*{Outline of the proof of Theorem \ref{thm1-bdy}}
 Before stating it, we denote  
\begin{equation}
	\label{xi-alpha}
\begin{aligned}
\xi_1, \cdots, \xi_{n-1}  
\end{aligned}
\end{equation} 
an orthonormal basis of 
$T^{1,0}_{\partial M}:=T^{1,0}_{\bar M} \cap T^{\mathbb{C}}_{ \partial M}$,
 $\nu$ the unit inner normal vector along the boundary, and 
\begin{equation}
	\label{xi-n}
\begin{aligned}
 \xi_n=\frac{1}{\sqrt{2}}\left(\mathrm{{\bf \nu}}-\sqrt{-1}J\nu\right). 
\end{aligned}
\end{equation}

\subsubsection*{Double normal case}
Under the assumption  
either $\partial M$ satisfies 
\eqref{bdry-assumption1} or $f$ satisfies
 \eqref{unbounded},
we derive 
the  boundary estimate for double normal derivative
in the following quantitative form
\begin{equation}
\label{yuan-prop1}
\begin{aligned}
 \mathfrak{g}(\xi_n, J\bar \xi_n)(p_0) \leq C\left(1 +  \sum_{\alpha=1}^{n-1} |\mathfrak{g}(\xi_\alpha, J\bar \xi_n)(p_0)|^2\right),
 \mbox{  } \forall p_0\in\partial M. 
\end{aligned}
\end{equation}
 
\noindent{\bf Case 1: $\partial M$ satisfies  \eqref{bdry-assumption1}.}
  This assumption 
   is used to compare $\mathfrak{g}_{\alpha\bar\beta}$ with  
  $\underline{\mathfrak{g}}_{\alpha\bar\beta}$ when restricted to boundary. We observe that it enables us to directly apply Lemmas 
  \ref{yuan's-quantitative-lemma} and \ref{lemma3.4}
   to understand the quantitative version of boundary estimate for double normal derivative (see Proposition \ref{proposition-quar-yuan1}). 
 
 \noindent{\bf Case 2: $f$ satisfies \eqref{unbounded}.}
    Without imposing any restriction to shape of boundary, following a strategy of Caffarelli-Nirenberg-Spruck \cite{CNS3}, we construct delicate barrier functions based on a characterization of $\Gamma_\infty$ to compare $\mathfrak{g}_{\alpha\bar\beta}$ with $\underline{\mathfrak{g}}_{\alpha\bar\beta}$ on boundary.
  Subsequently, we apply Lemmas 
 \ref{lemma3.4} and \ref{yuan's-quantitative-lemma} 
 to prove  \eqref{yuan-prop1} in 
 Proposition \ref{proposition-quar-yuan2}.


 \subsubsection*{Tangential-Normal case}  
From the quantitative boundary estimate \eqref{yuan-prop1} for double normal derivative, it requires to prove quantitative boundary estimate for tangential-normal derivatives
\begin{equation}
\begin{aligned}
  |\mathfrak{g}(\xi_\alpha,J\bar\xi_n)(p_0)| \leq C\left(1+\sup_M|\nabla u|\right), \mbox{  } \forall 1\leq\alpha\leq n-1, \mbox{  } \forall p_0\in\partial M. 
\end{aligned}
\end{equation}
This is proved in Proposition \ref{mix-general}, using local barrier technique 
going back at least to \cite{Hoffman1992Boundary,Guan1993Boundary,Guan1998The}.  



 
 \subsection*{Sketch the proof of gradient estimate}
 
 The main obstruction
  to deriving \eqref{estimate00} is 
  the gradient estimate.
 Our strategy is blow-up argument.  
According to Lemma \ref{lemma3.4} below, in the presence of \eqref{elliptic} and \eqref{concave},
 condition \eqref{addistruc} 
  is in effect equivalent to 
 \begin{equation}
 	\label{addistruc-0}
 		\lim_{t\rightarrow +\infty}f(t\lambda)>f(\mu), \mbox{ } \forall \lambda, \mbox{ }\mu\in \Gamma.
 \end{equation}
 Such a condition 
 allows one to apply the blow-up argument in \cite[Section 6]{Gabor}. 
 
 As shown in Theorem \ref{thm1-bdy}, the {admissible} solutions satisfy 
  \begin{equation}
  	\begin{aligned}
  		\sup_{\partial M} \Delta u \leq C\left(1+\sup_M |\nabla u|^2\right).  \nonumber
  	\end{aligned}
  \end{equation}
On the other hand,
Sz\'ekelyhidi's second  estimate
 \cite{Gabor} 
 yields that any  {admissible} solution to \eqref{mainequ}   satisfies
 \begin{equation}
 	\label{quantitative-2nd-boundary-estimate}
 	\begin{aligned}
 		\sup_{ M} \Delta u
 		\leq C\left(1+ \sup_{M}|\nabla u|^{2} +\sup_{\partial M}|\Delta u|\right).   \nonumber   
 	\end{aligned}
 \end{equation}
Moreover, by $\Gamma\subset\Gamma_1$, one has $\Delta u> -\mathrm{tr}_\omega \chi.$
 With those at hand, we 
 obtain \eqref{sec-estimate-quar1}, i.e.
 \begin{equation}	
 	\begin{aligned}
	\sup_{M}\Delta u\leq C \left(1+\sup_{M} |\nabla u|^2\right). \nonumber
\end{aligned} \end{equation}
 Consequently,   the gradient estimate can be derived from \eqref{sec-estimate-quar1},   using the Liouville type theorem \cite[Theorem 20]{Gabor}.

 \section{Lemmas}
 \label{sec3}
 
 To prove 
 Theorem \ref{thm1-bdy}, we propose 
  Lemmas \ref{lemma3.4} and \ref{yuan's-quantitative-lemma}. 
 
 \subsection{Criteria for symmetric concave functions}
 The concavity of $f$ implies
 \begin{equation}\label{concavity1}\begin{aligned}
\sum_{i=1}^n f_i(\lambda)(\mu_i-\lambda_i)\geq f(\mu)-f(\lambda), \mbox{  }\forall\lambda, \mbox{  }\mu\in\Gamma.
\end{aligned}\end{equation}
This yields
 \begin{equation}\label{010} \begin{aligned}
 \mbox{For any } \lambda, \mbox{  } \mu\in\Gamma,  \mbox{  } \sum_{i=1}^n f_i(\lambda)\mu_i\geq  \limsup_{t\rightarrow+\infty} f(t\mu)/t. \nonumber
 \end{aligned}\end{equation}
 Inspired by this observation,
we introduce the following condition 
 \begin{equation}
\label{addistruc-1}
\begin{aligned}
\mbox{For any } \lambda\in\Gamma, \mbox{  } 
\limsup_{t\rightarrow+\infty} f(t\lambda)/t\geq0.
\end{aligned}
\end{equation}
This leads to
\begin{lemma}
\label{lemma-new-1}
For $f$ satisfying  \eqref{concave}, the condition \eqref{addistruc-1} is equivalent to each one of the following three conditions
 \begin{equation}
 \label{addistruc-2}
 \begin{aligned}
 \sum_{i=1}^n f_i(\lambda)\mu_i\geq 0,  \mbox{ } \forall\lambda, \mbox{  } \mu\in\Gamma,
   \end{aligned}
 \end{equation}
\begin{equation}
	\label{addistruc-10-2}   \begin{aligned}
		f(\lambda+\mu)\geq f(\lambda), \mbox{ }\forall \lambda, \mbox{  } \mu\in \Gamma, 
\end{aligned} \end{equation}
 \begin{equation}
 \label{addistruc-3}
 \begin{aligned}
   \sum_{i=1}^n f_i(\lambda)\lambda_i\geq0, \mbox{   } \forall \lambda\in\Gamma. 
  \end{aligned}
 \end{equation}

 If, in addition, $\sum_{i=1}^n f_i(\lambda)>0$
 then \eqref{addistruc-1} is equivalent to each one of the following:
  \begin{equation}
 \label{addistruc-4}
 \begin{aligned}
 \sum_{i=1}^n f_i(\lambda)\mu_i>0, \mbox{   } \forall \lambda, \mbox{  } \mu\in \Gamma,
  \end{aligned}
 \end{equation}
\begin{equation}\label{addistruc-10}   \begin{aligned}
		f(\lambda+\mu)>f(\lambda), \mbox{ }\forall \lambda, \mbox{  } \mu\in \Gamma. 
\end{aligned} \end{equation}
\end{lemma}
\begin{proof}
Obviously, \eqref{addistruc-1} implies \eqref{addistruc-2}.
From \eqref{concavity1}, we have 
\begin{equation}
	\label{hhh-1}
	\begin{aligned}
		f(\lambda+\mu)-f(\lambda)\geq \sum_{i=1}^n f_i(\lambda+\mu)\mu_i.
	\end{aligned}
\end{equation}
Thus \eqref{addistruc-2} implies \eqref{addistruc-10-2}. It follows from \eqref{addistruc-10-2} that
\begin{equation}
	\label{hhh-2}
		\begin{aligned}
	f(t\lambda)\geq f(s\lambda), \mbox{ } \forall \lambda\in\Gamma, \mbox{  } \forall t>s
		\end{aligned}
\end{equation}
which means $\frac{d}{d t}f(t\lambda)\geq 0$. Thus 
 \eqref{addistruc-10-2}  yields \eqref{addistruc-3}.
 By \eqref{addistruc-3} we have \eqref{hhh-2}. And then it gives \eqref{addistruc-1}.  

 By the openness of $\Gamma$, for $\mu\in\Gamma$ there is  $\delta_\mu>0$ such that
  $\mu-\delta_\mu \vec{\bf1}\in\Gamma$. 
  By \eqref{addistruc-2}, we have
$\sum_{i=1}^n f_i(\lambda)\mu_i\geq \delta_{\mu}\sum_{i=1}^n f_i(\lambda)$ and then \eqref{addistruc-4} if $\sum_{i=1}^n f_i(\lambda)>0$. 
By \eqref{hhh-1} and \eqref{addistruc-4}, we derive \eqref{addistruc-10}.
\end{proof}

  	

 We now give   
 criteria of symmetric concave functions satisfying \eqref{addistruc}.
\begin{lemma}
\label{lemma3.4}
In the presence of \eqref{elliptic} and \eqref{concave}, the following statements are equivalent.
\begin{enumerate}
\item[{\bf(1)}] $f$ satisfies  \eqref{addistruc}.
\item[{\bf(2)}] $f$ satisfies  \eqref{addistruc-0}.
\item[{\bf(3)}] $f$ satisfies  \eqref{addistruc-1}.
\item[{\bf(4)}] $f$ satisfies \eqref{addistruc-2}.
\item[{\bf(5)}] $f$ satisfies \eqref{addistruc-10-2}.
\item[{\bf(6)}] $f$ satisfies  \eqref{addistruc-3}.
\item[{\bf(7)}] $f$ satisfies \eqref{addistruc-4}.
\item[{\bf(8)}] $f$ satisfies \eqref{addistruc-10} 
\end{enumerate}
\end{lemma}

\begin{proof}
Obviously, \eqref{addistruc-0}$\Rightarrow$\eqref{addistruc}$\Rightarrow$\eqref{addistruc-1}$\Leftrightarrow$\eqref{addistruc-2}$\Leftrightarrow$\eqref{addistruc-4}$\Leftrightarrow$\eqref{addistruc-10}$\Leftrightarrow$\eqref{addistruc-10-2}$\Leftrightarrow$\eqref{addistruc-3}$\Rightarrow$\eqref{addistruc}. 

It requires only to prove \eqref{addistruc-10}$\Rightarrow$\eqref{addistruc-0}.
Fix $\lambda$, $\mu\in\Gamma$. Since $\Gamma$ is open, $t\lambda-\mu=t(\lambda-\mu/t)\in \Gamma$ for $t\geq t_{\lambda,\mu}$, depending only on $\lambda$ and $\mu$. 
Thus \eqref{addistruc-10}  yields
\[f(t\lambda)=f(\mu+(t\lambda-\mu))>f(\mu).\]
\end{proof}

\begin{remark}
	Together with \cite[Lemma 6.2]{CNS3} (a  special case of a result of  \cite{Marcus1956}),
Lemma \ref{lemma3.4} implies that for any $n\times n$ Hermitian matrices 
 $A=(A_{i\bar j})$, $B=(B_{i\bar j})$ 
 with $\lambda(A)$, $\lambda(B)\in\Gamma$, 
 \begin{equation}
 \label{key-01-yuan3}
    \begin{aligned}
    \frac{\partial F}{\partial A_{i\bar j}} (A)B_{i\bar j}>0. 
    \end{aligned}
    \end{equation}
Together with \eqref{addistruc-10}, we have
 \begin{equation}
 \label{concavity2}
    \begin{aligned}
    F(A+B)>F(A).
    \end{aligned}
    \end{equation}
      
 \end{remark}

 \begin{remark}
	Condition \eqref{addistruc-10} and \eqref{concavity2} play a vital role in proof of quantitative boundary estimate for pure normal derivative.
\end{remark}

 \begin{remark}
According to Lemma \ref{lemma3.4}, 
   if 
\eqref{concave} and \eqref{addistruc} hold then
\begin{equation}
\label{sumfi}
\begin{aligned}
\sum_{i=1}^n f_i(\lambda)>\frac{f(R\vec{\bf 1})-f(\lambda)}{R} \mbox{ for any } R>0.  \nonumber   
\end{aligned}
\end{equation}
In particular, there is a uniform positive constant $\kappa_\sigma$ depending on $\sigma$ such that 
\begin{equation}
\label{sumfi1}
\begin{aligned}
\sum_{i=1}^n f_i(\lambda) \geq\kappa_\sigma \mbox{ in }   \partial\Gamma^\sigma. 
\end{aligned}
\end{equation}
 
 \end{remark}

  \subsection{Quantitative lemmas}
\label{refinementofCNS}

A key ingredient for quantitative boundary
for double normal derivative is how to follow the track of the behavior of  the eigenvalues of the matrix $\left(\mathfrak{g}(\xi_i, J\bar \xi_j)\right)$
 as $\mathfrak{g}(\xi_n, J\bar \xi_n)$
tends to infinity. To this end, we prove the following lemma.

\begin{lemma}
	\label{yuan's-quantitative-lemma}
	Let $A$ be an $n\times n$ Hermitian matrix
	\begin{equation}\label{matrix3}\left(\begin{matrix}
			d_1&&  &&a_{1}\\ &d_2&& &a_2\\&&\ddots&&\vdots \\ && &  d_{n-1}& a_{n-1}\\
			\bar a_1&\bar a_2&\cdots& \bar a_{n-1}& \mathrm{{\bf a}} \nonumber
		\end{matrix}\right)\end{equation}
	with $d_1,\cdots, d_{n-1}, a_1,\cdots, a_{n-1}$ fixed, and with $\mathrm{{\bf a}}$ variable.
	Denote the eigenvalues of $A$ by $\lambda=(\lambda_1,\cdots, \lambda_n)$.
	Let $\epsilon>0$ be a fixed constant.
	Suppose that  the parameter $\mathrm{{\bf a}}$ 
	satisfies the quadratic
	growth condition  
	\begin{equation}
		\begin{aligned}
			\label{guanjian1-yuan}
			\mathrm{{\bf a}}\geq \frac{2n-3}{\epsilon}\sum_{i=1}^{n-1}|a_i|^2 +(n-1)\sum_{i=1}^{n-1} |d_i|+ \frac{(n-2)\epsilon}{2n-3}.
		\end{aligned}
	\end{equation}
	Then the eigenvalues 
	(possibly with a proper permutation)
	behave like
	\begin{equation}
		\begin{aligned}
		 d_{\alpha}-\epsilon 	\,& < 
			\lambda_{\alpha} < d_{\alpha}+\epsilon, \mbox{  } \forall 1\leq \alpha\leq n-1, \\ \nonumber
	 \mathrm{{\bf a}} 	\,& \leq \lambda_{n}
			< \mathrm{{\bf a}}+(n-1)\epsilon. \nonumber
		\end{aligned}
	\end{equation}
\end{lemma}

This lemma asserts that  
if
the parameter 
$\mathrm{{\bf a}}$  
satisfies the quadratic growth condition \eqref{guanjian1-yuan}
then  the eigenvalues concentrate 
near   
diagonal elements 
correspondingly. Consequently, it suggests an effective way to  follow the track of the behavior of  the eigenvalues as $|\mathrm{{\bf a}}|$ tends to infinity.  
 In fact, Lemma \ref{yuan's-quantitative-lemma} can be viewed as a quantitative version of  \cite[Lemma 1.2]{CNS3}.
\begin{lemma}
[{\cite[Lemma 1.2]{CNS3}}]
\label{lemmaCNS3}
Consider the $n\times n$ symmetric matrix
\begin{equation}
\label{matrix1}
A=\left(
\begin{matrix}
d_1&&  &&a_{1}\\
&d_2&& &a_2\\
&&\ddots&&\vdots \\
&& &  d_{n-1}& a_{n-1}\\
a_1&a_2&\cdots& a_{n-1}& \mathrm{{\bf a}}              \nonumber
\end{matrix}
\right)
\end{equation}
with $d_1,\cdots, d_{n-1}$ fixed, $| \mathrm{{\bf a}} |$ tends to infinity and
\begin{equation}
|a_i|\leq C, \mbox{ } i=1,\cdots, n. \nonumber
\end{equation}
Then the eigenvalues $\lambda_1,\cdots, \lambda_{n}$ behave like
\begin{equation}
\label{behave1}
\begin{aligned}
\,&
\lambda_{\alpha}=d_{\alpha}+o(1),\mbox{  } 1\leq \alpha \leq n-1,\\
\,&
\lambda_{n}=\mathrm{{\bf a}}\left(1+O\left(1/\mathrm{{\bf a}}\right)\right),      \nonumber
\end{aligned}
\end{equation}
where the $o(1)$ and $O(1/\mathrm{{\bf a}})$ are uniform--depending only on $d_{1},\cdots, d_{n-1}$ and $C$.
\end{lemma}

This lemma was first used by Caffarelli-Nirenberg-Spruck  \cite{CNS3}  and later by \cite{Trudinger95,LiSY2004} to derive boundary
estimates for double normal derivative of
certain fully nonlinear elliptic equations on bounded domains $\Omega$ in Euclidean spaces.
However, their boundary estimate is not a quantitative form. The reason is that  Lemma \ref{lemmaCNS3} does not figure out
how the eigenvalues of matrix $A$ concentrate explicitly near the corresponding
 diagonal elements  when $|\mathrm{{\bf a}}|$ is sufficiently large. 
%


\vspace{1mm}
In the rest of this subsection, we complete the proof of Lemma \ref{yuan's-quantitative-lemma}.
We start with the case $n=2$. 
For $n=2$, the eigenvalues of $A$ are
 $$\lambda_{1}=\frac{\mathrm{{\bf a}}+d_1- \sqrt{(\mathrm{{\bf a}}-d_1)^2+4|a_1|^2}}{2},
\quad \lambda_2=\frac{\mathrm{{\bf a}}+d_1+\sqrt{(\mathrm{{\bf a}}-d_1)^2+4|a_1|^2}}{2}.$$
We assume $a_1\neq 0$; otherwise we are done.
If $\mathrm{{\bf a}} \geq \frac{|a_1|^2}{ \epsilon}+ d_1$ then one has
\begin{equation}
\begin{aligned}
0\leq d_1- \lambda_1 =\lambda_2-\mathrm{{\bf a}}
= \frac{2|a_1|^2}{\sqrt{ (\mathrm{{\bf a}}-d_1)^2+4|a_1|^2 } +(\mathrm{{\bf a}}-d_1)}
< \frac{|a_1|^2}{\mathrm{{\bf a}}-d_1 } \leq \epsilon.   \nonumber
\end{aligned}
\end{equation}
Here we use $a_1\neq 0$ to confirm the strictly inequality in the above formula.

The following lemma enables us to count  the eigenvalues near the diagonal elements
via a deformation argument. It is an essential  ingredient in the proof of
 Lemma \ref{yuan's-quantitative-lemma} for general $n$.
\begin{lemma}
\label{refinement}
Let $A$ be 
a Hermitian $n$ by  $n$  matrix
\begin{equation}\label{matrix2}\left(\begin{matrix}d_1&&  &&a_{1}\\&d_2&& &a_2\\&&\ddots&&\vdots \\&& &  
d_{n-1}& a_{n-1}\\ \bar a_1&\bar a_2&\cdots& \bar a_{n-1}& \mathrm{{\bf a}} \nonumber
\end{matrix}\right)\end{equation} with $d_1,\cdots, d_{n-1}, a_1,\cdots, a_{n-1}$ fixed, and with $\mathrm{{\bf a}}$ variable.
Denote
$\lambda=(\lambda_1,\cdots, \lambda_n)$ as the the eigenvalues of $A$ with the order
$\lambda_1\leq \lambda_2 \leq\cdots \leq \lambda_n$.
Fix a positive constant $\epsilon$.
Suppose that the parameter $\mathrm{{\bf a}}$ in the matrix $A$ satisfies  the following quadratic growth condition
\begin{equation}
\label{guanjian2}
\begin{aligned}
\mathrm{{\bf a}} \geq \frac{1}{\epsilon}\sum_{i=1}^{n-1} |a_i|^2+\sum_{i=1}^{n-1}  \left(d_i+ (n-2) |d_i|\right)+ (n-2)\epsilon.
\end{aligned}
\end{equation}
Then for any $\lambda_{\alpha}$ $(1\leq \alpha\leq n-1)$ there exists a $d_{i_{\alpha}}$
with lower 
index $1\leq i_{\alpha}\leq n-1$ such that
\begin{equation}
\label{meishi}
\begin{aligned}
 |\lambda_{\alpha}-d_{i_{\alpha}}|<\epsilon,
\end{aligned}
\end{equation}
\begin{equation}
\label{mei-23-shi}
0\leq \lambda_{n}-\mathrm{{\bf a}} <(n-1)\epsilon + \left|\sum_{\alpha=1}^{n-1}(d_{\alpha}-d_{i_{\alpha}})\right|.
\end{equation}
\end{lemma}

\begin{proof}
Without loss of generality, we assume $\sum_{i=1}^{n-1} |a_i|^2>0$ and  $n\geq 3$
(otherwise we are done). 
Note that  the eigenvalues have
the order $\lambda_1\leq \lambda_2\leq \cdots \leq \lambda_n$, as in the assumption of lemma.
It is well known that, 
for a Hermitian matrix, any diagonal element is   less than or equals to   the  largest eigenvalue.
 In particular,
 \begin{equation}
 \label{largest-eigen1}
 \lambda_n \geq \mathrm{{\bf a}}.
 \end{equation}

We only need to prove   \eqref {meishi}, since  \eqref{mei-23-shi} is a consequence of  \eqref{meishi}, \eqref{largest-eigen1}  and
\begin{equation}
\label{trace}
 \sum_{i=1}^{n}\lambda_i=\mbox{tr}(A)=\sum_{\alpha=1}^{n-1} d_{\alpha}+\mathrm{{\bf a}}.
 \end{equation}

 Let's denote   $I=\{1,2,\cdots, n-1\}$. We divide the index set   $I$ into two subsets:
$${\bf B}=\{\alpha\in I: |\lambda_{\alpha}-d_{i}|\geq \epsilon, \mbox{   }\forall i\in I \}, $$
$$ {\bf G}=I\setminus {\bf B}=\{\alpha\in I: \mbox{There exists $i\in I$ such that } |\lambda_{\alpha}-d_{i}| <\epsilon\}.$$

To complete the proof, it only requires to prove ${\bf G}=I$ or equivalently ${\bf B}=\emptyset$.
  It is easy to see that  for any $\alpha\in {\bf G}$, one has
   \begin{equation}
   \label{yuan-lemma-proof1}
   \begin{aligned}
   |\lambda_\alpha|< \sum_{i=1}^{n-1}|d_i| + \epsilon.
   \end{aligned}
   \end{equation}

   Fix $ \alpha\in {\bf B}$,  we are going to give the estimate for $\lambda_\alpha$.
The eigenvalue $\lambda_\alpha$ satisfies
\begin{equation}
\label{characteristicpolynomial}
\begin{aligned}
(\lambda_{\alpha} -\mathrm{{\bf a}})\prod_{i=1}^{n-1} (\lambda_{\alpha}-d_i)
= \sum_{i=1}^{n-1} |a_{i}|^2 \prod_{j\neq i} (\lambda_{\alpha}-d_{j}).
\end{aligned}
\end{equation}
By the definition of ${\bf B}$, for  $\alpha\in {\bf B}$, one then has $|\lambda_{\alpha}-d_i|\geq \epsilon$ for any $i\in I$.
We therefore derive
\begin{equation}
\begin{aligned}
|\lambda_{\alpha}-\mathrm{{\bf a}} |\leq \sum_{i=1}^{n-1} \frac{|a_i|^2}{|\lambda_{\alpha}-d_{i}|}\leq
\frac{1}{\epsilon}\sum_{i=1}^{n-1} |a_i|^2, \quad \mbox{ if } \alpha\in {\bf B}.
\end{aligned}
\end{equation}
Hence,  for $\alpha\in {\bf B}$, we obtain
\begin{equation}
\label{yuan-lemma-proof2}
\begin{aligned}
 \lambda_\alpha \geq \mathrm{{\bf a}}-\frac{1}{\epsilon}\sum_{i=1}^{n-1} |a_i|^2.
\end{aligned}
\end{equation}

For a set ${\bf S}$, we denote $|{\bf S}|$ the  cardinality of ${\bf S}$.
We shall use proof by contradiction to prove  ${\bf B}=\emptyset$.
Assume ${\bf B}\neq \emptyset$.
Then $|{\bf B}|\geq 1$, and so $|{\bf G}|=n-1-|{\bf B}|\leq n-2$. 

We compute the trace of the matrix $A$ as follows:
\begin{equation}
\begin{aligned}
\mbox{tr}(A)=\,&
\lambda_n+
\sum_{\alpha\in {\bf B}}\lambda_{\alpha} + \sum_{\alpha\in  {\bf G}}\lambda_{\alpha}\\
> \,&
\lambda_n+
|{\bf B}| (\mathrm{{\bf a}}-\frac{1}{\epsilon}\sum_{i=1}^{n-1} |a_i|^2 )-|{\bf G}| (\sum_{i=1}^{n-1}|d_i|+\epsilon ) \\
\geq \,&
 2\mathrm{{\bf a}}-\frac{1}{\epsilon}\sum_{i=1}^{n-1} |a_i|^2 -(n-2) (\sum_{i=1}^{n-1}|d_i|+\epsilon )
\\
\geq \,& \sum_{i=1}^{n-1}d_i +\mathrm{{\bf a}}= \mbox{tr}(A),
\end{aligned}
\end{equation}
where we use  \eqref{guanjian2},   \eqref{largest-eigen1}, \eqref{yuan-lemma-proof1} and \eqref{yuan-lemma-proof2}.
This is a contradiction.
We have ${\bf B}=\emptyset$.
Therefore, ${\bf G}=I$ and  the proof is complete.
\end{proof}

We consequently obtain

\begin{lemma}
\label{refinement111}
Let $A(\mathrm{{\bf a}})$ be an $n\times n$ Hermitian  matrix
\begin{equation}
A(\mathrm{{\bf a}})=\left(
\begin{matrix}
d_1&&  &&a_{1}\\
&d_2&& &a_2\\
&&\ddots&&\vdots \\
&& &  d_{n-1}& a_{n-1}\\
\bar a_1&\bar a_2&\cdots& \bar a_{n-1}& \mathrm{{\bf a}} \nonumber
\end{matrix}
\right)
\end{equation}
with
 $d_1,\cdots, d_{n-1}, a_1,\cdots, a_{n-1}$ fixed, and with $\mathrm{{\bf a}}$ variable.
 Assume that $d_1, d_2, \cdots, d_{n-1} $
 are distinct with each other, i.e. $d_{i}\neq d_{j}, \forall i\neq j$.
Denote
$\lambda=(\lambda_1,\cdots, \lambda_n)$ as the the eigenvalues of $A(\mathrm{{\bf a}})$.
Given a positive constant $\epsilon$ with
$0<\epsilon \leq \frac{1}{2}\min\{|d_i-d_j|: \forall i\neq j\}$, 
if the parameter  $\mathrm{{\bf a}}$ satisfies the quadratic growth condition 
\begin{equation}
 \begin{aligned}
\label{guanjian1}
\mathrm{{\bf a}}\geq \frac{1}{\epsilon}\sum_{i=1}^{n-1}|a_i|^2 +(n-1)\sum_{i=1}^{n-1} |d_i|+(n-2)\epsilon,
\end{aligned}
\end{equation}
then the eigenvalues behave like
\begin{equation}
\begin{aligned}
 \,& |d_{\alpha}-\lambda_{\alpha}|
<   \epsilon, \forall 1\leq \alpha\leq n-1,\\
\,& 0\leq \lambda_{n}-\mathrm{{\bf a}}
<  (n-1)\epsilon.  \nonumber
\end{aligned}
\end{equation}
\end{lemma}

\begin{proof}
The proof is based on Lemma  \ref{refinement} and a deformation argument.
Without loss of generality,  we assume $n\geq 3$ and  $\sum_{i=1}^{n-1} |a_i|^2>0$
 (otherwise  
  we are done).
 Moreover, we assume in addition that $d_1<d_2\cdots<d_{n-1}$ and the eigenvalues have the order
 $$\lambda_1\leq \lambda_2\cdots \leq \lambda_{n-1}\leq \lambda_n.$$

  Fix $\epsilon\in (0,\mu_0]$, where $\mu_0=\frac{1}{2}\min\{|d_i-d_j|: \forall i\neq j\}$.
  We denote
   $$I_{i}=(d_i-\epsilon,d_i+\epsilon)$$
     and $$P_0=\frac{1}{\epsilon}\sum_{i=1}^{n-1} |a_i|^2+ (n-1)\sum_{i=1}^{n-1} |d_i|+ (n-2)\epsilon.$$
    Since $0<\epsilon \leq \mu_0$,  the intervals disjoint each other
  \begin{equation}
\label{daqin122}
I_\alpha\bigcap I_\beta=\emptyset \mbox{ for }  1\leq \alpha<\beta\leq n-1.
\end{equation}

In what follows,
 we assume that  the parameter  $\mathrm{{\bf a}}$ satisfies \eqref{guanjian1} and  the Greek letters $\alpha,  \beta$
  range from $1$ to $n-1$.
We define a function
$$\mathrm{{\bf Card}}_\alpha: [P_0,+\infty)\rightarrow \mathbb{N}$$
 to count the eigenvalues which lie in $I_\alpha$.
 (Note that when the eigenvalues are not distinct,  the function $\mathrm{{\bf Card}}_\alpha$ means the summation of all the algebraic
  multiplicities
 of  distinct eigenvalues 
which  lie in $I_\alpha$). 
  This function measures the number of the  eigenvalues which lie in $I_\alpha$.

 We are going to prove that $\mathrm{{\bf Card}}_\alpha$ is continuous on $[P_0,+\infty)$ in an attempt to
 complete the proof.

 First Lemma \ref{refinement} asserts that if $\mathrm{{\bf a}} \geq P_0$, then
 \begin{equation}
\label{daqin111}
\begin{aligned}
\lambda_{\alpha}\in  \bigcup_{i=1}^{n-1} I_i, \mbox{   } \forall 1\leq \alpha\leq n-1.
\end{aligned}
\end{equation}

 It is well known that the largest eigenvalue $\lambda_n\geq \mathrm{{\bf a}}$,
 while the smallest eigenvalue $\lambda_1 \leq d_1$.
 Combining it with    \eqref{daqin111}
one has
 \begin{equation}
 \label{largest1}
\begin{aligned}
\lambda_{n} \geq  \mathrm{{\bf a}}
>\,& \sum_{i=1}^{n-1}|d_i| +\epsilon.
\end{aligned}
\end{equation}
Thus $\lambda_n\in \mathbb{R}\setminus (\bigcup_{i=1}^{n-1} \overline{I_i})$
where $\overline{I_i}$ denotes the closure of $I_i$.
Therefore, 
the function $\mathrm{{\bf Card}}_\alpha$
is continuous (and so it is constant), since  
 \eqref{daqin111}, \eqref{daqin122},  $\lambda_n\in \mathbb{R}\setminus (\overline{\bigcup_{i=1}^{n-1} I_i})$
and  the eigenvalues of $A(\mathrm{{\bf a}})$   depend  on the parameter  $\mathrm{{\bf a}}$ continuously.

The continuity of $\mathrm{{\bf Card}}_\alpha(\mathrm{{\bf a}})$ plays a crucial role in this proof.
Following the line of the proof of  Lemma \ref{lemmaCNS3} 
(\cite[Lemma 1.2]{CNS3}),
in the setting of Hermitian matrices,  one can show that for $1\leq \alpha\leq n-1$,
\begin{equation}
\label{yuanrr-lemma-12321}
\lim_{\mathrm{{\bf a}}\rightarrow +\infty} \mathrm{{\bf Card}}_{\alpha}(\mathrm{{\bf a}}) \geq 1.
\end{equation}
It follows from \eqref{largest1},  \eqref{yuanrr-lemma-12321}
 and the  continuity  of $\mathrm{{\bf Card}}_\alpha $ that
  \begin{equation}
\label{daqin142224}
\mathrm{{\bf Card}}_{\alpha}(\mathrm{{\bf a}})= 1,
 \mbox{   } \forall \mathrm{{\bf a}}\in [P_0, +\infty),  \mbox{   } 1\leq \alpha\leq n-1.  \nonumber
\end{equation}
Together with \eqref{daqin111}, we prove that,
for any   $1\leq \alpha\leq n-1$, the interval $I_\alpha=(d_{\alpha}-\epsilon,d_{\alpha}+\epsilon)$
contains the eigenvalue $\lambda_\alpha$.
We thus complete  the proof.
\end{proof}




 Suppose that there are two distinct indices $i_0, j_0$ ($i_0\neq j_0$) such that $d_{i_{0}}= d_{j_{0}}$.
  Then the characteristic polynomial of $A$
can be rewritten as the following
\begin{equation}
\begin{aligned}
(\lambda-d_{i_{0}})\left[(\lambda-\mathrm{{\bf a}})\prod_{i \neq i_{0}} (\lambda-d_i)-|a_{i_{0}}|^{2}\prod_{j\neq j_0, j\neq i_0} (\lambda-d_j)
-\sum_{i\neq i_0}|a_i|^2\prod_{j\neq i, j\neq i_{0}} (\lambda-d_j)\right]. \nonumber
\end{aligned}
\end{equation}
So $\lambda_{i_0}=d_{i_{0}}$ is an eigenvalue of $A$ for any $a\in \mathbb{R}$.
Noticing that the following polynomial
$$(\lambda-\mathrm{{\bf a}})\prod_{i \neq i_{0}} (\lambda-d_i)
-|a_{i_{0}}|^{2}\prod_{j\neq j_0, j\neq i_0} (\lambda-d_j)
-\sum_{i\neq i_0}|a_i|^2\prod_{j\neq i, j\neq i_{0}} (\lambda-d_j)$$
 is the characteristic polynomial of the $(n-1)\times (n-1)$ Hermitian  matrix
\begin{equation}
\left(
\begin{matrix}
d_1&&  &&&&a_{1}\\
&\ddots && &&&\vdots \\
&&\widehat{d_{i_{0}}}&& &&\widehat{a_{i_{0}}}\\
&&&\ddots&&&\vdots \\
&&& &  d_{j_{0}}&& (|a_{j_{0}}|^{2}+|a_{i_{0}}|^{2})^{\frac{1}{2}}\\
&&&& &   \ddots              \\
\bar a_1&\cdots&\widehat{ \bar a_{i_{0}}}&\cdots& (|a_{j_{0}}|^{2}+|a_{i_{0}}|^{2})^{\frac{1}{2}} &\cdots & \mathrm{{\bf a}}   \nonumber
\end{matrix}
\right)
\end{equation}
where $\widehat{*}$ indicates deletion.
Therefore, $(\lambda_1, \cdots, \widehat{\lambda_{i_{0}}}, \cdots, \lambda_{n})$ are the eigenvalues of the above $(n-1)\times (n-1)$
Hermitian  matrix.
Hence,  we obtain 
\begin{lemma}
\label{refinement3}
Let $A$ be as in Lemma \ref{yuan's-quantitative-lemma} an $n\times n$ Hermitian matrix.
Let
\[\mathcal{I}=
\begin{cases}
\mathbb{R}^{+}=(0,+\infty) \,& \mbox{ if } d_{i}=d_{1}, \forall 2\leq i \leq n-1;\\
 \left(0,\mu_0\right),  \mbox{  } \mu_0=\frac{1}{2}\min\{|d_{i}-d_{j}|: d_{i}\neq d_{j}\} \,& \mbox{ otherwise.}
\end{cases}
\]
Denote
$\lambda=(\lambda_1,\cdots, \lambda_n)$ by the the eigenvalues of $A$. Fix  $\epsilon\in \mathcal{I}$.
Suppose that  the parameter $\mathrm{{\bf a}}$ in $A$ satisfies  \eqref{guanjian1}.
 Then the eigenvalues behave like
\begin{equation}
\begin{aligned}
\,& |d_{\alpha}-\lambda_{\alpha}|
< \epsilon, \mbox{    } \forall 1\leq \alpha\leq n-1,\\
\,&0 \leq \lambda_{n}-\mathrm{{\bf a}}
< (n-1)\epsilon. \nonumber
\end{aligned}
\end{equation}
\end{lemma}

Applying Lemmas \ref{refinement} and \ref{refinement3},
we  complete the proof of Lemma \ref{yuan's-quantitative-lemma}  without restriction to the applicable
 scope of $\epsilon$.

\begin{proof}
[Proof of Lemma \ref{yuan's-quantitative-lemma}]
We follow the outline of the proof of Lemma \ref{refinement111}.
Without loss of generality, we may assume
 $$n\geq 3, \mbox{ } \sum_{i=1}^{n-1} |a_i|^2>0,  \mbox{ }  d_1\leq d_2\leq \cdots \leq d_{n-1}  \mbox{ and } \lambda_1\leq \lambda_2 \leq \cdots \lambda_{n-1}\leq \lambda_n.$$

Fix $\epsilon>0$.  Let $I'_\alpha=(d_\alpha-\frac{\epsilon}{2n-3}, d_\alpha+\frac{\epsilon}{2n-3})$ and
$$P_0'=\frac{2n-3}{\epsilon}\sum_{i=1}^{n-1} |a_i|^2+ (n-1)\sum_{i=1}^{n-1} |d_i|+ \frac{(n-2)\epsilon}{2n-3}.$$
  In what follows we assume \eqref{guanjian1-yuan} holds.
The connected components of $\bigcup_{\alpha=1}^{n-1} I_{\alpha}'$ are denoted as in the following:
$$J_{1}=\bigcup_{\alpha=1}^{j_1} I_\alpha', \mbox{ }
J_2=\bigcup_{\alpha=j_1+1}^{j_2} I_\alpha', \mbox{ }  \cdots, J_i =\bigcup_{\alpha=j_{i-1}+1}^{j_i} I_\alpha', \mbox{ } \cdots,
 J_{m} =\bigcup_{\alpha=j_{m-1}+1}^{n-1} I_\alpha'.$$
 Moreover,
   \begin{equation}
   \begin{aligned}
J_i\bigcap J_k=\emptyset, \mbox{ for }   1\leq i<k\leq m. \nonumber
\end{aligned}
\end{equation}
It plays formally the role of \eqref{daqin122} in the proof of Lemma \ref{refinement111}.

 As in the proof of Lemma \ref{refinement111},
 we let
  $$ \mathrm{{\bf \widetilde{Card}}}_k:[P_0',+\infty)\rightarrow \mathbb{N}$$
be the function that counts the eigenvalues which lie in $J_k$.
   (Note that when the eigenvalues are not distinct,  the function $\mathrm{{\bf \widetilde{Card}}}_k$ denotes  the summation of all the algebraic
    multiplicities of distinct eigenvalues which
 lie in $J_k$).
By Lemma \ref{refinement} and  $$\lambda_n \geq {\bf a}\geq P_0'>\sum_{i=1}^{n-1}|d_i|+\frac{\epsilon}{2n-3}$$ 
we conclude that
 if the parameter $\mathrm{{\bf a}}$ satisfies the quadratic growth condition \eqref{guanjian1-yuan} then
   \begin{equation}
  \label{yuan-lemma-proof5}
  \begin{aligned}
   \,& \lambda_n \in \mathbb{R}\setminus (\bigcup_{k=1}^{n-1} \overline{I_k'})
   =\mathbb{R}\setminus (\bigcup_{i=1}^m \overline{J_i}), \\
  \,& \lambda_\alpha \in \bigcup_{i=1}^{n-1} I_{i}'=\bigcup_{i=1}^m J_{i} \mbox{ for } 1\leq\alpha\leq n-1.
  \end{aligned}
  \end{equation}
Similarly, $\mathrm{{\bf \widetilde{Card}}}_i(\mathrm{{\bf a}})$ is a continuous function
 with respect to the variable $\mathrm{{\bf a}}$ when ${\bf a}\geq P_0'$. So it is a constant.
 Combining it with Lemma \ref{refinement3}, 
 we see that 
$$ \mathrm{{\bf \widetilde{Card}}}_i(\mathrm{{\bf a}})
=j_i-j_{i-1}$$
for ${\bf a}\geq P_0'$. Here we denote $j_0=0$ and $j_m=n-1$.
We thus know that the   $(j_i-j_{i-1})$ eigenvalues
$$\lambda_{j_{i-1}+1}, \lambda_{j_{i-1}+2}, \cdots, \lambda_{j_i}$$
lie in the connected component $J_{i}$.
Thus, for any $j_{i-1}+1\leq \gamma \leq j_i$,  we see $I_\gamma'\subset J_i$ and  $\lambda_\gamma$
   lies in the connected component $J_{i}$.
Therefore,
$$|\lambda_\gamma-d_\gamma| < \frac{(2(j_i-j_{i-1})-1) \epsilon}{2n-3}\leq \epsilon.$$
Here we use the fact that $d_\gamma$ is the midpoint of  $I_\gamma'$ and $J_i\subset \mathbb{R}$ is an open subset.
\end{proof}

   \section{Quantitative boundary estimate: double normal case}
  \label{sec4}
  
This section is devoted to deriving quantitative version of boundary estimate for double normal derivative. 
The key ingredients in the proof 
are Lemmas   \ref{lemma3.4} and \ref{yuan's-quantitative-lemma}.
 
\begin{proposition}
\label{proposition-quar-yuan1}
Let $(M, J,\omega)$ be a compact Hermitian manifold with $C^2$ boundary satisfying \eqref{bdry-assumption1}.  Let $\xi_i$ be as in \eqref{xi-alpha} and \eqref{xi-n}.
Let $u\in  C^{2}(\bar M)$ be an \textit{admissible} solution
 to Dirichlet problem \eqref{mainequ}. Suppose  \eqref{elliptic}-\eqref{addistruc}, 
 \eqref{existenceofsubsolution} and \eqref{nondegenerate} hold. Then 
\begin{equation}
\begin{aligned}
  \mathfrak{g}(\xi_n, J\bar \xi_n)(p_0)
  \leq C\left(1 +  \sum_{\alpha=1}^{n-1} |\mathfrak{g}(\xi_\alpha, J\bar \xi_n)(p_0)|^2\right),  \quad \forall p_0\in\partial M, \nonumber
\end{aligned}
\end{equation}
where $C$ is a uniform positive constant  depending  only on   $|u|_{C^0(\bar M)}$, 
$|\underline{u}|_{C^{2}(\bar M)}$, $\partial M$ up to second order derivatives 
 and other known data
 (but neither on $\sup_{M}|\nabla u|$ nor on $(\delta_{\psi,f})^{-1}$).
\end{proposition}

Without any restriction \eqref{bdry-assumption1} to the shape of boundary, 
 we obtain
the following proposition when $f$ satisfies the unbounded condition \eqref{unbounded}.
\begin{proposition}
\label{proposition-quar-yuan2}
Let $(M, J,\omega)$ be a compact Hermitian manifold with $C^3$ boundary,  let $\xi_i$ be as in \eqref{xi-alpha} and \eqref{xi-n}. 
In addition to \eqref{elliptic}-\eqref{addistruc}, 
 \eqref{existenceofsubsolution} and \eqref{nondegenerate}, we assume $f$ satisfies the unbounded condition \eqref{unbounded}.
Then  any \textit{admissible} solution $u\in  C^{2}(\bar M)$ to Dirichlet problem \eqref{mainequ} satisfies 
\begin{equation}
	\begin{aligned}
		\mathfrak{g}(\xi_n, J\bar \xi_n)(p_0)
		\leq C\left(1 +  \sum_{\alpha=1}^{n-1} |\mathfrak{g}(\xi_\alpha, J\bar \xi_n)(p_0)|^2\right),  \quad \forall p_0\in\partial M.\nonumber
	\end{aligned}
\end{equation}
 Here the constant $C$ depends on  $(\delta_{\psi,f})^{-1}$,  $\sup_M\psi$, $|u|_{C^0(\bar M)}$,  
$|\underline{u}|_{C^{2}(\bar M)}$, $\partial M$ up to third order derivatives 
 and other known data
 (but not on $\sup_{M}|\nabla u|$).
 
\end{proposition}
     
    \subsection{Preliminaries}

  Given $p_0\in\partial M$, we can choose a local holomorphic coordinate systems 
  \begin{equation}
 \label{goodcoordinate1}
\begin{aligned}
(z_1,\cdots, z_n), \mbox{  } z_i=x_i+\sqrt{-1}y_i, 
\end{aligned}
\end{equation}
 centered at $p_0$,
 so that $g_{i\bar j}(0)=\delta_{ij}$, $\frac{\partial}{\partial x_n}$ is the inner normal vector at origin, and
   $T^{1,0}_{{p_0},{\partial M}}$ is spanned by $\frac{\partial}{\partial z_\alpha}$ for $1\leq\alpha\leq n-1$.
   For convenience we set
\begin{equation}
\label{t-coordinate}
\begin{aligned}
t_{2k-1}=x_{k}, \ t_{2k}=y_{k},\ 1\leq k\leq n-1;\ t_{2n-1}=y_{n},\ t_{2n}=x_{n}.  
\end{aligned}
\end{equation}
   We also denote $\sigma(z)$ by the distance function from $z$ to $\partial M$ with respect to $\omega$. 
 Near the origin $p_0$, 
 \begin{equation}
 \label{asym-sigma1}
 \begin{aligned}
 \sigma(z)=x_n+\sum_{ i,j=1}^{2n} a_{ij} t_it_j+O(|t|^3).
 \end{aligned}
 \end{equation}

  In the computation we use derivatives with respect to  Chern connection $\nabla$ of $\omega$,
and write
$\partial_{i}=\frac{\partial}{\partial z_{i}}$, 
$\overline{\partial}_{i}=\frac{\partial}{\partial \bar z_{i}}$,
 $\nabla_{i}=\nabla_{\frac{\partial}{\partial z_{i}}}$,
 $\nabla_{\bar i}=\nabla_{\frac{\partial}{\partial \bar z_{i}}}$.
 For   a smooth function $v$,
$$
v_i:=
\partial_i v,  
\mbox{  } v_{\bar i}:=
\partial_{\bar i} v,
\mbox{  } 
v_{i\bar j}:= 
\partial_i\overline{\partial}_j v, 
\mbox{  } 
 v_{ij}:=
 \nabla_{j}\nabla_{i} v =
 \partial_i \partial_j v -\Gamma^k_{ji}v_k, \cdots \mbox{etc},
$$
where $\Gamma_{ij}^k$ are the Christoffel symbols
 defined  by 
$\nabla_{\frac{\partial}{\partial z_i}} \frac{\partial}{\partial z_j}=\Gamma_{ij}^k \frac{\partial}{\partial z_k}.$
 
 Now we derive the $C^0$-estimate, and  boundary gradient  estimates.
 Let $\check{u}$ be the solution to 
 \begin{equation}
 \label{supersolution-1}
     \begin{aligned}
  \,& \Delta \check{u} + \mathrm{tr}_\omega(\chi)=0 \mbox{ in } M,   \,& \check{u}=\varphi \mbox{ on } \partial M.
 \end{aligned}
\end{equation}
The existence of $\check{u}$ follows from standard theory of elliptic equations of second order. Such $\check{u}$ is a supersolution of \eqref{mainequ}.
By the maximum principle and boundary value condition,  one derives
    \begin{equation}
     \label{key-14-yuan3}
     \begin{aligned}
     \underline{u}_{x_n}(0) \leq u_{x_n}(0) \leq \check{u}_{x_n}(0), \mbox{ }  u_\alpha(0)=\underline{u}_\alpha(0), \mbox{  } \underline{u}\leq u\leq \check{u} \mbox{ in } M.
     \end{aligned}
\end{equation}
This simply gives the following lemma.
\begin{lemma}
	\label{lemma-c0-bc1}
There is a uniform positive constant $C$ such that 
\begin{equation}
     \begin{aligned}
   \sup_M |u|+\sup_{\partial M} |\nabla u| \leq C.
     \end{aligned}
\end{equation}
\end{lemma}

 \subsection{First ingredient of the proof}
In the proof the Greek letters  $\alpha,  \beta$
  range from $1$ to $n-1$.
    Let's denote
 \begin{equation}
A(R)=\left(
\begin{matrix}
\mathfrak{g}_{1\bar 1}&\mathfrak{g}_{1\bar 2}&\cdots &\mathfrak{g}_{1\overline{(n-1)}} &\mathfrak{g}_{1\bar n}\\
\mathfrak{g}_{2\bar 1} &\mathfrak{g}_{2\bar 2}&\cdots& \mathfrak{g}_{2\overline{(n-1)}}&\mathfrak{g}_{2\bar n}\\
\vdots&\vdots&\ddots&\vdots&\vdots \\
\mathfrak{g}_{(n-1)\bar 1}&\mathfrak{g}_{(n-1)\bar 2}& \cdots&  \mathfrak{g}_{{(n-1)}\overline{(n-1)}}&
 \mathfrak{g}_{(n-1)\bar n}\\
\mathfrak{g}_{n\bar 1}&\mathfrak{g}_{n\bar 2}&\cdots& \mathfrak{g}_{n \overline{(n-1)}}& R  \nonumber
\end{matrix}
\right).
\end{equation}
 
 Let $\lambda=\lambda(\mathfrak{g})$, $\underline{\lambda}
 =\lambda(\mathfrak{\underline{g}})$,  $\lambda'=\lambda_{\omega'}(\mathfrak{g}_{\alpha\bar\beta})$,  
 $\underline{\lambda}'=\lambda_{\omega'}(\underline{\mathfrak{g}}_{\alpha\bar\beta})$. Here as in Theorem \ref{thm1-2} we denote $\omega^{\prime}=\left.\omega\right|_{T_{\partial M} \cap J T_{\partial M}}$. We know
    \begin{equation}
    \label{yuan3-31}
    \begin{aligned}
    \lambda', \mbox{  } \underline{\lambda}' \in \Gamma_\infty.
    \end{aligned}
    \end{equation}

The boundary value condition implies 
  \begin{equation}
  \label{yuan3-buchong5}
    \begin{aligned}
    u_{\alpha\bar\beta}(0)=\underline{u}_{\alpha\bar\beta}(0)+(u-\underline{u})_{x_n}(0)\sigma_{\alpha\bar\beta}(0).
    \end{aligned}
    \end{equation}
    Let $\eta=(u-\underline{u})_{x_n}(0)$, then at $p_0$ ($z=0$)
     \begin{equation}
    \begin{aligned}
        \mathfrak{g}_{\alpha\bar\beta}=\underline{\mathfrak{g}}_{\alpha\bar\beta}+ \eta\sigma_{\alpha\bar\beta}.
    \end{aligned}
    \end{equation}

         We rewrite 
         $\mathfrak{g}_{\alpha\bar\beta}$ as 
    \begin{equation}
    \begin{aligned}
        \mathfrak{g}_{\alpha\bar\beta}= (1-t)\underline{\mathfrak{g}}_{\alpha\bar\beta}
        +\left(t\underline{\mathfrak{g}}_{\alpha\bar\beta}+\eta\sigma_{\alpha\bar\beta}\right).
    \end{aligned}
    \end{equation}
For simplicity, 
 we denote
     \begin{equation}
     	\label{A_t}
    \begin{aligned}
    A_t=\sqrt{-1} \left[t\underline{\mathfrak{g}}_{\alpha\bar\beta}+\eta\sigma_{\alpha\bar\beta}\right]dz_\alpha\wedge d\bar z_\beta.
    \end{aligned}
    \end{equation}
    Clearly, $(A_{1})_{\alpha\bar\beta}=\mathfrak{g}_{\alpha\bar\beta}$ so  $\lambda_{\omega'}(A_1)\in\Gamma_\infty$. 
    
    \begin{lemma}
    \label{yuan-k2v}
    Suppose there are constants $t_0<1$ and $R_0>0$  such that 
    \begin{equation}
    	\label{yuanrr-412}
\begin{aligned}
\lambda_{\omega'}(A_{t_0})\in \overline{\Gamma}_\infty,
    \end{aligned}
\end{equation}
     \begin{equation}   \label{key-03-yuan3v}\begin{aligned}
 f\left((1-t_0)\underline{\lambda}'_{1},\cdots, (1-t_0)\underline{\lambda}'_{n-1}, R_0\right)\geq f(\underline{\lambda}).
\end{aligned}\end{equation}
 Then there is a uniform positive constant $C$ depending on  $(1-t_0)^{-1}$, $|t_0|$, $R_0$ and other known data such that 
   \begin{equation}
   \begin{aligned}
   \mathfrak{g}_{n\bar n}\leq C\left(1+\sum_{\alpha=1}^{n-1}|\mathfrak{g}_{\alpha\bar n}|^2\right).  \nonumber
   \end{aligned}
   \end{equation}
    \end{lemma}
    
   \begin{proof}
   	 It follows from \eqref{key-03-yuan3v} and the openness of $\Gamma$  that 
   	 there are two uniform positive constants $R_1$   and $\varepsilon_0$   such that
   	\begin{equation}
   		\label{key-03-yuan3}
   		\begin{aligned}
   			f\left((1-t_0)(\underline{\lambda}'_{1}-{\varepsilon_0}/{2}),\cdots, (1-t_0)(\underline{\lambda}'_{n-1}
   			-{\varepsilon_0}/{2}), R_1\right)\geq f(\underline{\lambda}),
   		\end{aligned}
   	\end{equation}
   	\begin{equation}
   		\label{key-002-yuan3}
   		\begin{aligned}
   			\left(\underline{\lambda}'_1-\varepsilon_{0}, \cdots, \underline{\lambda}'_{n-1}-\varepsilon_0, {R_1}/{(1-t_0)}\right)\in \Gamma,
   		\end{aligned}
   	\end{equation}
   	where $R_1$ depends on  $R_0$, $\underline{\lambda}'$, and possibly on $(1-t_0)^{-1}$, 
   	and $\varepsilon_0$ depends on 
   	$\inf_{\partial M}\mathrm{dist}(\underline{\lambda},\partial \Gamma)$, $\underline{\lambda}'$ and other known data.
   	
   	Let $$ {A}(R) =
   	 \left( \begin{matrix}
   	 \mathfrak{{g}}_{\alpha\bar \beta} &\mathfrak{g}_{\alpha\bar n}\\
   		\mathfrak{g}_{n\bar \beta}& R  \nonumber
   	 \end{matrix}\right).$$  
    By \eqref{A_t} we can decompose $A(R)$ into
   \begin{equation}
   	\begin{aligned}
 {A}(R) 
= A'(R)+A''(R) 
\end{aligned}
\end{equation}
where 
 \[A'(R)=\left(\begin{matrix}
	(1-t_0)(\mathfrak{\underline{g}}_{\alpha\bar \beta}-\frac{\varepsilon_0}{4}\delta_{\alpha\beta}) &\mathfrak{g}_{\alpha\bar n}\\
	\mathfrak{g}_{n\bar \beta}& R/2  \nonumber
\end{matrix}\right), \quad
A''(R)=\left(\begin{matrix} (A_{t_0})_{\alpha\bar \beta}+\frac{(1-t_0)\varepsilon_0}{4} \delta_{\alpha\beta} &0\\ 0& R/2  \end{matrix}\right).\]

   Denote
   \begin{equation}
   	\label{yuan3-buchong3}
   	\lambda_{\omega'}(A_{t_0}):=
   	\tilde{\lambda}'=(\tilde{\lambda}_1',\cdots,\tilde{\lambda}_{n-1}').
   \end{equation}
   By \eqref{key-14-yuan3} there is a uniform constant $C_0>0$ depending on $|t_0|$,
   $\sup_{\partial M}|\nabla u|$ and other known data, such that $|\tilde{\lambda}'|\leq C_0$, that is
   $\tilde{\lambda'}$ is contained in a compact subset of $\overline{\Gamma}_\infty$, i.e.  
   \begin{equation}
   	\label{yuan3-buchong4v}
   	\begin{aligned}
   		\tilde{\lambda}'\in {K}:=\{\lambda'\in \overline{\Gamma}_\infty: |\lambda'|\leq C_0\}.  \nonumber
   	\end{aligned}
   \end{equation}
   Combining with \eqref{yuanrr-412} there is a 
   uniform positive constant $R_2$ depending on $((1-t_0)\varepsilon_0)^{-1}$, $K$ and other known data, 
   such that  
   \begin{equation}
   	\label{at0-1}
   	\begin{aligned}
\lambda(A''(R))\in\Gamma, \mbox{ } \forall R>R_2.
\end{aligned}
   \end{equation}

Let's pick  $\epsilon=\frac{(1-t_0)\varepsilon_0}{4}$ in 
 Lemma  \ref{yuan's-quantitative-lemma}, and we set
\begin{equation}
\begin{aligned}
R_c
= \,& \frac{8(2n-3)}{(1-t_0)\varepsilon_0}\sum_{\alpha=1}^{n-1} | \mathfrak{g}_{\alpha\bar n}|^2 
  +\frac{2(2n^2-4n+1)(1-t_0)\varepsilon_0}{4(2n-3)}
  + 2(n-1)  |{\underline{\lambda}'}|
  +2R_1+2R_2  \nonumber
\end{aligned}
\end{equation}
where $\varepsilon_0$, $R_1$ and $R_2$ are fixed constants as we have chosen above.

According to Lemma \ref{yuan's-quantitative-lemma},
 the eigenvalues $\lambda({A}'(R_c))$ of ${A}'(R_c)$
   (possibly with an appropriate order)
 shall behave like
\begin{equation}
\label{lemma12-yuan}
\begin{aligned}
\,& \lambda_{\alpha}({A}'(R_c))\geq (1-t_0)(\underline{\lambda}'_{\alpha}-\frac{\varepsilon_0}{2}), \mbox{  } 1\leq \alpha\leq n-1, \\\,&
\lambda_{n}({A}'(R_c))\geq R_c/2-(n-1)(1-t_0)\varepsilon_0/4.
\end{aligned}
\end{equation}
In particular, $\lambda({A}'(R_c))\in \Gamma$. So $\lambda(A(R_c))\in \Gamma$. 
 
 Next we will use Lemma \ref{lemma3.4}. 
 Precisely, together with \eqref{at0-1},
 \eqref{concavity2} gives
\begin{equation}
\label{yuan-k1}
    \begin{aligned}
    f(\lambda(A(R_c)))\geq f(\lambda(A'(R_c))). 
    \end{aligned}
    \end{equation}
 From \eqref{key-03-yuan3}, \eqref{lemma12-yuan} and \eqref{yuan-k1}, we deduce   
 \begin{equation}
   	\label{yuan-k2}
   	\begin{aligned}
   		\mathfrak{g}_{n\bar n} \leq R_c.   \nonumber
   	\end{aligned}
   \end{equation}
 
\end{proof}

\subsection{Second ingredient of the proof}

 In order to prove Propositions \ref{proposition-quar-yuan1} and \ref{proposition-quar-yuan2}, according to Lemma \ref{yuan-k2v}, 
 it requires only to complete the following two steps:
 \begin{itemize}
 	\item Confirm the assumptions imposed in Lemma \ref{yuan-k2v}.
 	\item  Prove that $(1-t_0)^{-1}$ can be uniformly
 	 bounded from above, i.e., 
 	\begin{equation}
 		\label{bound-t0}
 		\begin{aligned}
 			(1-t_0)^{-1}\leq C. 
 		\end{aligned}
 	\end{equation}
 \end{itemize}   

\subsubsection*{\bf Case 1: $\partial M$ satisfies \eqref{bdry-assumption1}.}
Note that $\eta=(u-\underline{u})_{x_n}(0)\geq 0$.
The assumption  \eqref{bdry-assumption1} yields 
$\lambda_{\omega'}(\eta\sigma_{\alpha\bar\beta})\in \overline{\Gamma}_\infty$ on $\partial M$. 
Since $\underline{u}$ is a subsolution, one can choose a constant $R$ sufficiently large, such that
\begin{equation}
	\begin{aligned}
		f(\underline{\lambda}',\mathfrak{\underline{g}}_{n\bar n})\geq f(\underline{\lambda}). \nonumber
	\end{aligned}
\end{equation}
 Here we use \eqref{concavity1}. 
So the $t_0$ in  Lemma \ref{yuan-k2v} exists and
  $$t_0=0.$$
Thus \eqref{bound-t0} automatically holds. 
 Consequently, we can deduce Proposition \ref{proposition-quar-yuan1}.

\subsubsection*{\bf Case 2: $f$ satisfies the unbounded condition \eqref{unbounded}.}

This case corresponds exactly to Proposition \ref{proposition-quar-yuan2}.
 We always assume $\Gamma$ is of type 1 in the sense of Caffarelli-Nirenberg-Spruck \cite{CNS3}, then $\Gamma_\infty$ is an open symmetric convex cone in $\mathbb{R}^{n-1}$ and $\Gamma_\infty\neq\mathbb{R}^{n-1}$;
 otherwise, $\Gamma_\infty=\mathbb{R}^{n-1}$, then we have done as shown in Case 1. 
 
 We first confirm the assumptions of Lemma  \ref{yuan-k2v}.
 \begin{itemize}
 \item  
Let $t_0$ be the first 
  $t$ as we decrease $t$ from $1$ 
  such that
      \begin{equation}
      \label{key0-yuan3}
    \begin{aligned}
    \lambda_{\omega'}(A_{t_0})\in\partial\Gamma_\infty.
    \end{aligned}
    \end{equation}
 Such $t_0$ exists, since 
  $\lambda_{\omega'}(A_1)\in\Gamma_\infty$ and
  $\lambda_{\omega'}(A_t)\in \mathbb{R}^{n-1}\setminus\Gamma_\infty$ for $t\ll -1$. Furthermore, for a uniform positive constant $T_0$ under control,
    \begin{equation}
    \label{1-yuan3}
    \begin{aligned}
   -T_0< t_0<1.
    \end{aligned}
    \end{equation}
    
 \item  By the unbounded condition \eqref{unbounded}
    there is a uniform positive constant $R_1$ depending on $(1-t_0)^{-1}$, 
    $\varepsilon_0$ and $\underline{\lambda}'$
    such that \eqref{key-03-yuan3} holds.
Here is the only place where we use the unbounded condition \eqref{unbounded}. 
As a contrast,  such an unbounded condition  can be removed when $t_0=0$ as in Case 1.

  \end{itemize}
    
   To complete the proof of Proposition \ref{proposition-quar-yuan2}, it requires only to prove the following lemma.
    \begin{lemma}
    \label{keylemma1-yuan3}
    Let $t_0$ be as defined in \eqref{key0-yuan3}, then 
    	\begin{equation}
    	\begin{aligned}
    		(1-t_0)^{-1}\leq C  \nonumber
    	\end{aligned}
    \end{equation}
where $C$ is a uniform positive constant depending on 
    $|u|_{C^0(M)}$, $|\nabla u|_{C^0(\partial M)}$, $|\underline{u}|_{C^2(M)}$, $\sup_M\psi$,
    $(\delta_{\psi,f})^{-1}$, $\partial M$ up to third derivatives and other known data.
      
\end{lemma}

We assume $\eta>0$ (otherwise we have done).  
In the proof we shall make use of some idea of Caffarelli-Nirenberg-Spruck \cite{CNS3}, which was used by Li \cite{LiSY2004} 
to study the Dirichlet problem in $\mathbb{C}^n$.
We use some notation of \cite{CNS3,LiSY2004}. 
   Without loss of generality, we assume $$t_0>\frac{1}{2} \mbox{ and } \tilde{\lambda}_1'\leq \cdots \leq \tilde{\lambda}_{n-1}'$$ 
   where $\tilde{\lambda}'$ is as we denoted in \eqref{yuan3-buchong3}.
   Combining \eqref{1-yuan3} with \eqref{key-14-yuan3}, we can deduce that there is a uniform constant $C_0'>0$ depending on 
   $\sup_{\partial M}|\nabla u|$ and other known data, such that $|\tilde{\lambda}'|\leq C_0'$, 
    i.e.  
 \begin{equation}
 \label{yuan3-buchong4}
\begin{aligned}
\tilde{\lambda}'\in {K}':=\{\lambda'\in\partial\Gamma_\infty: |\lambda'|\leq C_0'\}.
\end{aligned}
\end{equation}

    It was proved in \cite[Lemma 6.1]{CNS3} that for $\tilde{\lambda}'\in\partial\Gamma_\infty$
     there is a supporting plane for
     $\Gamma_\infty$ and one can choose $\mu_j$ with 
   $\mu_1\geq \cdots\geq \mu_{n-1}\geq0$ so that
    \begin{equation}
    \label{key-18-yuan3}
    \begin{aligned}
    \Gamma_\infty\subset \left\{\lambda'\in\mathbb{R}^{n-1}: \sum_{\alpha=1}^{n-1}\mu_\alpha\lambda'_\alpha>0 \right\}, \mbox{  }
  \mbox{  } \sum_{\alpha=1}^{n-1} \mu_\alpha=1, \mbox{  } \sum_{\alpha=1}^{n-1}\mu_\alpha \tilde{\lambda}_\alpha'=0.
    \end{aligned}
    \end{equation} 
  According to \cite[Lemma 6.2]{CNS3}
(without loss of generality we assume $\underline{\lambda}_1'\leq \cdots\leq\underline{\lambda}_{n-1}'$),
   \begin{equation}
    \begin{aligned}
    \sum_{\alpha=1}^{n-1} \mu_\alpha \underline{\mathfrak{g}}_{\alpha\bar\alpha}\geq \sum_{\alpha=1}^{n-1}\mu_\alpha \underline{\lambda}'_\alpha\geq \inf_{p\in\partial M}\sum_{\alpha=1}^{n-1} \mu_\alpha \underline{\lambda}'_\alpha(p)\geq a_0>0.
    \end{aligned}
    \end{equation}
    Here we use \eqref{yuan3-31}, \eqref{yuan3-buchong4} and \eqref{key-18-yuan3}. We shall mention that $a_0$ depends on $\mathrm{disc}(\underline{\lambda},\partial \Gamma)$. 
   Without loss of generality,  we assume $({A_{t_0}})_{\alpha\bar\beta}=t_0\underline{\mathfrak{g}}_{\alpha\bar\beta}+\eta\sigma_{\alpha\bar\beta}$ is diagonal at $p_0$. From \eqref{key-18-yuan3} one has at the origin
     \begin{equation}
    \begin{aligned}
  0=  t_0 \sum_{\alpha=1}^{n-1}\mu_\alpha\underline{\mathfrak{g}}_{\alpha\bar\alpha}+\eta\sum_{\alpha=1}^{n-1}\mu_\alpha  {\sigma}_{\alpha\bar\alpha} 
  \geq a_0t_0 +\eta \sum_{\alpha=1}^{n-1} \mu_\alpha \sigma_{\alpha\bar\alpha}.
    \end{aligned}
    \end{equation}
 Together with \eqref{key-14-yuan3}, we see at the origin $\{z=0\}$
    \begin{equation}
    \label{key-1-yuan3}
    \begin{aligned}
    -\sum_{\alpha=1}^{n-1} \mu_\alpha \sigma_{\alpha\bar\alpha}\geq \frac{a_0 t_0}{\sup_{\partial M}|\nabla (\check{u}-\underline{u})|}=:a_1>0,
    \end{aligned}
    \end{equation}
    where $\check{u}$ and $\underline{u}$ are respectively supersolution and subsolution.
         Let $$\Omega_\delta=M\cap B_{\delta}(0),$$
where $B_\delta(0)=\{z\in M: |z|<\delta\}$.
On $\Omega_\delta$, we let
\begin{equation}
    \begin{aligned}
    d(z)=\sigma(z)+\tau |z|^2
    \end{aligned}
    \end{equation}
    where $\tau$ is a positive constant 
     to be determined; and let 
    \begin{equation}
     \label{w-buchong1}
    \begin{aligned}
    w(z)=\underline{u}(z)+({\eta}/{t_0})\sigma(z)+l(z)\sigma(z)+Ad(z)^2,
    \end{aligned}
    \end{equation}
    where $l(z)=\sum_{i=1}^n (l_iz_i+\bar l_{i}\bar z_{i})$, 
    $l_i\in \mathbb{C}$, $\bar l_i=l_{\bar i}$,
     to be chosen as in \eqref{chosen-1}, and $A$ is a positive constant to be determined.
       Furthermore, on $\partial M\cap\bar\Omega_\delta$, $u(z)-w(z)=-A\tau^2|z|^4$.
    On $M\cap\partial B_{\delta}(0)$,
    \begin{equation}
    \begin{aligned}
    u(z)-w(z) 
    \leq \,& |u-\underline{u}|_{C^0(\Omega_\delta)} 
    -(2A\tau \delta^2+\frac{\eta}{t_0}-2n \sup_{i}|l_i| \delta)\sigma(z)-A\tau^2\delta^4 \\
    \leq \,& -\frac{A\tau^2 \delta^4}{2} \nonumber
    \end{aligned}
    \end{equation}
 provided $A\gg1$. 

Let $T_1(z),\cdots, T_{n-1}(z)$ be an orthonormal basis for holomorphic tangent space  of level hypersurface
 $\{w: d(w)=d(z)\}$ at $z$, so that at the origin 
 $T_\alpha(0)= \frac{\partial }{\partial z_\alpha}$  for each $1\leq\alpha\leq n-1$.

  Such a basis exists: 
  We see at the origin  $\partial d(0)=\partial \sigma(0)$.  
  Thus for $1\leq \alpha\leq n-1$, we can choose $T_\alpha$ such that at the  origin
  $T_\alpha(0)= \frac{\partial }{\partial z_\alpha}$.

 By   \cite[Lemma 6.2]{CNS3}, we have the following lemma.
\begin{lemma}
\label{lemma-yuan3-buchong1}
Let $T_1(z),\cdots, T_{n-1}(z)$ be as above, and  let  $T_n=\frac{\partial d}{|\partial d|}$.
For a real $(1,1)$-form $\Theta=\sqrt{-1}\Theta_{i\bar j}dz_i\wedge d\bar z_j$,
we denote by $\lambda(\Theta)=(\lambda_1(\Theta),\cdots,\lambda_n(\Theta))$  the eigenvalues of $\Theta$ (with respect to $\omega$) with $\lambda_1(\Theta)\leq \cdots\leq \lambda_n(\Theta)$. Then for any 
$\mu_1\geq\cdots\geq\mu_n$,
\[\sum_{i=1}^n \mu_i \lambda_{i}(\Theta)\leq \sum_{i=1}^n\mu_i\Theta(T_i,J\bar T_i).\]
\end{lemma}

Let $\mu_1,\cdots,\mu_{n-1}$ be as in \eqref{key-18-yuan3}, and set $\mu_n=0$. Let's denote $T_\alpha=\sum_{k=1}^nT_\alpha^k\frac{\partial }{\partial z_k}$. 
For $\Theta=\sqrt{-1}\Theta_{i\bar j}dz_i\wedge d\bar z_j$, we define
\begin{equation}
    \begin{aligned}
    \Lambda_\mu(\Theta):= \sum_{\alpha=1}^{n-1}\mu_{\alpha} T_{\alpha}^i \bar T_{\alpha}^j \Theta_{i\bar j}.  \nonumber
    \end{aligned}
    \end{equation}
    
    \begin{lemma}
    \label{lemma-key2-yuan3}
   There are parameters $\tau$, $A$, $l_i$, $\delta$ depending only on $|u|_{C^0(M)}$, 
   $|\nabla u|_{C^0(\partial M)}$,  
   $|\underline{u}|_{C^2(M)}$, 
   $\partial M$ up to third derivatives and other known data, such that 
   \begin{equation}
    \begin{aligned}
 \,&    \Lambda_\mu (\mathfrak{g}[w])   \leq0  \mbox{ in } \Omega_\delta, \,& u\leq w \mbox{ on } \partial \Omega_\delta. \nonumber
    \end{aligned}
    \end{equation}
    \end{lemma}
    
    \begin{proof}
 By direct computation
     \begin{equation}
    \begin{aligned}
    \Lambda_\mu (\mathfrak{g}[w])
    =\,&\sum_{\alpha=1}^{n-1} \mu_\alpha T_{\alpha}^i \bar T_{\alpha}^j
     (\chi_{i\bar j}+\underline{u}_{i\bar j}+\frac{\eta}{t_0}\sigma_{i\bar j}) 
       + 2Ad(z)\sum_{\alpha=1}^{n-1} \mu_\alpha T_{\alpha}^i \bar T_{\alpha}^j d_{i\bar j}
       \\\,&
       + \sum_{\alpha=1}^{n-1} \mu_\alpha T_{\alpha}^i \bar T_{\alpha}^j (l(z)\sigma_{i\bar j}
       +l_i\sigma_{\bar j}+\sigma_i l_{\bar j}).  \nonumber
    \end{aligned}
    \end{equation}
\begin{itemize}
     \item
At the origin $\{z=0\}$, $T_{\alpha}^i=\delta_{\alpha i}$, 
\begin{equation}
    \begin{aligned}
   \sum_{\alpha=1}^{n-1} \mu_\alpha T_{\alpha}^i \bar T_{\alpha}^j
     (\chi_{i\bar j}+\underline{u}_{i\bar j}+\frac{\eta}{t_0}\sigma_{i\bar j}) (0)
     =\frac{1}{t_0}\sum_{\alpha=1}^{n-1}\mu_\alpha (A_{t_0})_{\alpha\bar\alpha}=0.  \nonumber
    \end{aligned}
    \end{equation}
    So there are complex constants $k_i$ 
    such that
    \begin{equation}
    \begin{aligned}
   \sum_{\alpha=1}^{n-1} \mu_\alpha T_{\alpha}^i \bar T_{\alpha}^j
     (\chi_{i\bar j}+\underline{u}_{i\bar j}+\frac{\eta}{t_0}\sigma_{i\bar j}) (z)=\sum_{i=1}^n (k_i z_i+ \bar k_{i} \bar z_{i})+O(|z|^2). \nonumber
    \end{aligned}
    \end{equation}

\item   
    \begin{equation}
    \begin{aligned}
    2Ad(z)\sum_{\alpha=1}^{n-1} \mu_\alpha T_{\alpha}^i \bar T_{\alpha}^j d_{i\bar j} \leq -\frac{a_1A}{2}d(z). \nonumber
    \end{aligned}
    \end{equation}
    since  
     \begin{equation}
    \begin{aligned}
    \sum_{\alpha=1}^{n-1}\mu_\alpha T_{\alpha}^i \bar T_{\alpha}^j d_{i\bar j}
 =\,&   \sum_{\alpha=1}^{n-1} \mu_\alpha\sigma_{\alpha\bar\alpha}(z)+
  \tau\sum_{\alpha=1}^{n-1}\mu_\alpha
  \\\,&+\sum_{\alpha=1}^{n-1}\mu_\alpha \left(T_{\alpha}^i \bar T_{\alpha}^j (z)- T_{\alpha}^i \bar T_{\alpha}^j (0)\right)d_{i\bar j} \\
 =\,& -a_1+\tau+O(|z|) \leq -\frac{a_1}{4}  \nonumber
    \end{aligned}
    \end{equation}
    provided one chooses $0<\delta, \tau\ll1.$
Here we also use \eqref{key-1-yuan3},
    \begin{equation}
    \label{yuan3-buchong2}
    \begin{aligned}
    \sum_{\alpha=1}^{n-1} \mu_\alpha T_{\alpha}^i \bar T_{\alpha}^j(z)
    =\sum_{\alpha=1}^{n-1} \mu_\alpha T_{\alpha}^i \bar T_{\alpha}^j(0)+O(|z|)
    =\sum_{\alpha=1}^{n-1} \mu_\alpha \delta_{\alpha i} \delta_{\alpha j}+O(|z|),
       \end{aligned}
    \end{equation}
    and
\begin{equation}
     \label{c3-yuan3}
     \begin{aligned}
     \sum_{\alpha=1}^{n-1}\mu_\alpha\sigma_{\alpha\bar\alpha}(z)=\sum_{\alpha=1}^{n-1}\mu_\alpha\sigma_{\alpha\bar\alpha}(0)+O(|z|).
     \end{aligned}
     \end{equation}  
\item  
 \begin{equation}
    \begin{aligned}
  \,& 
     l(z)\sum_{\alpha=1}^{n-1}\mu_\alpha T^i_\alpha \bar T^j_\alpha
    \sigma_{i\bar j}+\sum_{\alpha=1}^{n-1}\mu_{\alpha}T^i_\alpha \bar T^j_\alpha(l_i\sigma_{\bar j}
    +\sigma_i l_{\bar j}) \\
   = \,&
    l(z) \sum_{\alpha=1}^{n-1}\mu_\alpha\sigma_{\alpha\bar\alpha}(0)
    - \tau\sum_{\alpha=1}^{n-1} \mu_\alpha (z_\alpha l_\alpha+\bar z_{\alpha}  \bar l_{\alpha})
    +O(|z|^2) \nonumber
    \end{aligned}
    \end{equation}
    since by \eqref{yuan3-buchong2},  
    $ \sum_{i=1}^n T_\alpha^i \sigma_i=-\tau\sum_{i=1}^n T_\alpha^i
\bar z_i $
we have
    \begin{equation}
    \begin{aligned}
     l(z) \sum_{\alpha=1}^{n-1}\mu_\alpha T^i_\alpha \bar T^j_\alpha \sigma_{i\bar j} 
    = l(z) \sum_{\alpha=1}^{n-1}\mu_\alpha\sigma_{\alpha\bar\alpha}(0)
+O(|z|^2),   \nonumber
    \end{aligned}
    \end{equation}
    \begin{equation}
    \begin{aligned}
       \sum_{\alpha=1}^{n-1}\mu_\alpha T^i_\alpha \bar T^j_\alpha(l_i \sigma_{\bar j} + \sigma_i l_{\bar j})
   =-\tau \sum_{\alpha=1}^{n-1} \mu_\alpha(\bar l_{\alpha} \bar z_{\alpha}+l_\alpha z_\alpha)
+O(|z|^2).   \nonumber
    \end{aligned}
    \end{equation}
    \end{itemize}
    Putting these together, 
     \begin{equation}
     \label{together1}
    \begin{aligned}
    \Lambda_\mu(\mathfrak{g}[w])\leq\,&
  \sum_{\alpha=1}^{n-1} 2\mathfrak{Re}
  \left\{z_\alpha(k_\alpha 
  -\tau\mu_\alpha l_\alpha
  +l_\alpha\sum_{\beta=1}^{n-1}\mu_\beta \sigma_{\beta\bar\beta}(0))\right\}
    \\ \,& + 2\mathfrak{Re}\left\{z_n(k_n+l_n\sum_{\beta=1}^{n-1}\mu_\beta \sigma_{\beta\bar\beta}(0))\right\} 
    -\frac{Aa_1}{2}d(z) + O(|z|^2).   \nonumber
    \end{aligned}
    \end{equation}
   Let $l_n=-\frac{k_n}{\sum_{\beta=1}^{n-1} \mu_\beta \sigma_{\beta\bar\beta}(0)}$.
  For $1\leq \alpha\leq n-1$, we set
  \begin{equation}
  \label{chosen-1}
    \begin{aligned}
    l_\alpha=-\frac{k_\alpha}{\sum_{\beta=1}^{n-1}\mu_\beta \sigma_{\beta\bar\beta}(0)-\tau \mu_\alpha}.
    \end{aligned}
    \end{equation}
From $\mu_\alpha\geq 0$ and \eqref{key-1-yuan3}, we see
    such $l_i$ (or equivalently the  $l(z)$) are all well defined and uniformly bounded.
    
    We thus complete the proof if $0<\tau, \delta\ll1$, $A\gg1$.
    \end{proof}

    \subsubsection*{\bf Completion of the proof of Lemma \ref{keylemma1-yuan3}}
  Let $w$ be as in Lemma \ref{lemma-key2-yuan3}. From the construction above, we know that there is a uniform positive constant $C_1'$ such that
  $$ |\mathfrak{g}[w]|_{C^0(\Omega_\sigma)}\leq C_1'.$$
    Let $\lambda[w]=\lambda_{\omega}(\mathfrak{g}[w])$. Assume 
    $\lambda_1[w]\leq \cdots\leq \lambda_n[w]$.
    Lemma \ref{lemma-key2-yuan3}, together with Lemma \ref{lemma-yuan3-buchong1}, implies
    \begin{equation}
    \begin{aligned}
    \sum_{\alpha=1}^{n-1} \mu_\alpha \lambda_\alpha[w]\leq 0 \mbox{ in } \Omega_\delta.  \nonumber
    \end{aligned}
    \end{equation}
So $(\lambda_1[w],\cdots,\lambda_{n-1}[w])\notin\Gamma_\infty$ by \eqref{key-18-yuan3}. In other words, $\lambda[w]\in X$, where
\[X=\{\lambda\in\mathbb{R}^n: \lambda'\in \mathbb{R}^{n-1}\setminus \Gamma_\infty\}\cap \{\lambda\in\mathbb{R}^n:   |\lambda|\leq C_1'\}.\]
Let $$\bar\Gamma^{\inf_M\psi}=\{\lambda\in\Gamma: f(\lambda)\geq \inf_M\psi\}$$
be the closure of sublevel set $\Gamma^{\inf_{M}\psi}$.
 Notice that $\Gamma_\infty$ is open so $X$ is a compact subset; furthermore $X\cap \bar\Gamma^{\inf_M\psi}=\emptyset$. 
So we can deduce that the distance between $\bar\Gamma^{\inf_M\psi}$ and $X$ 
is greater than some positive constant depending on  
$\delta_{\psi,f}$ and other known data.
Therefore, there exists an $\epsilon>0$  such that for any $z\in\Omega_\delta$
\begin{equation}
    \begin{aligned}
    \epsilon\vec{\bf 1} +\lambda[w]\notin \bar\Gamma^{\inf_M\psi}.  \nonumber
    \end{aligned}
    \end{equation}

By \eqref{asym-sigma1} one can choose a  positive constant $C'$ such that $ x_n\leq C'|z|^2$ on 
$\partial M\cap \bar{\Omega}_\delta$. As a result,  
there is a positive constant $C_2$ depending only on $M$ and $\delta$ so that  
$$x_n\leq C_2 |z|^2 \mbox{ on }\partial\Omega_\sigma.$$

 Let $\epsilon$ and $C_2$ be as above, we define ${h}(z)=w(z)+\epsilon (|z|^2-\frac{x_n}{C_2})$. Thus
    \begin{equation}
    \begin{aligned}
    u\leq {h} \mbox{ on } \partial \Omega_\delta.  \nonumber
    \end{aligned}
    \end{equation}
    Moreover, $\chi_{i\bar j}+{h}_{i\bar j}=(\chi_{i\bar j}+w_{i\bar j})+\epsilon\delta_{ij}$ so $\lambda[{h}]\notin \bar\Gamma^{\inf_M\psi}.$
    By \cite[Lemma B]{CNS3}, we have $$u\leq {h} \mbox{ in }\Omega_\delta.$$
   Notice 
   $u(0)=\varphi(0)$ and ${h}(0)=\varphi(0)$, we have $u_{x_n}(0)\leq h_{x_n}(0)$, and
    \begin{equation}
    \begin{aligned}
    t_0\leq \frac{1}{1+\epsilon/(\eta C_2)}, \mbox{ i.e., }  (1-t_0)^{-1}\leq 1+\frac{\eta C_2}{\epsilon}.  \nonumber
    \end{aligned}
    \end{equation}

 \section{Quantitative boundary estimate: tangential-normal case}
 \label{sec5}
 
 In this section we derive quantitative boundary estimate for tangential-normal derivatives.

 \begin{proposition}
\label{mix-general}
 
 Let $(M,J,\omega)$ be a compact Hermitian manifold with $C^3$-smooth boundary.
In addition we assume \eqref{elliptic}-\eqref{addistruc}, \eqref{existenceofsubsolution} and \eqref{nondegenerate} hold.
Then 
for any admissible solution $u\in C^3(M)\cap C^2(\bar M)$ to the Dirichlet problem \eqref{mainequ},
 there is a uniform positive constant $C$ depending on $|\varphi|_{C^{3}(\bar M)}$, 
 $|\underline{u}|_{C^{2}(\bar M)}$,  
$|\psi|_{C^1(M)}$, $|\nabla u|_{C^0(\partial M)}$,
$\partial M$ up to third derivatives
and other known data (but neither on $(\delta_{\psi,f})^{-1}$ nor on $\sup_{M}|\nabla u|$)
such that
\begin{equation}
\label{quanti-mix-derivative-00} 
|\nabla^2 u(T,\nu)|\leq C 
\left(1+\sup_{M}|\nabla u|\right)
 \end{equation}
 for any $T\in T_{\partial M}$ with $|T|=1$,   
where  $\nabla^2 u$ denotes the real Hessian of $u$.
\end{proposition}

 \subsection{Tangential operators on the boundary}
\label{Tangential-opera-1}
For  a  given point $p_0\in \partial M$,
we choose local holomorphic coordinates \eqref{goodcoordinate1}
  centered at $p_0$ in a neighborhood which we assume to be contained in $M_{\delta}:=\{z\in M: \sigma(z)<\delta\}$,
 such that $p_0=\{z=0\}$, $g_{i\bar j}(0)=\delta_{ij}$ and $\frac{\partial}{\partial x_{n}}$ is the interior normal direction to $\partial M$ at $p_0$.
As in \eqref{t-coordinate} we set
\begin{equation}
t_{2k-1}=x_{k}, \ t_{2k}=y_{k},\ 1\leq k\leq n-1;\ t_{2n-1}=y_{n},\ t_{2n}=x_{n}.  \nonumber
\end{equation}

We define the tangential operator on $\partial M$ 
 \begin{equation}
  \label{tangential-oper-general1}
\begin{aligned} 
 \mathcal{T}=\nabla_{\frac{\partial}{\partial t_{\alpha}}}- \widetilde{\eta}\nabla_{\frac{\partial}{\partial x_{n}}}, 
 \mbox{ for each fixed }
  1\leq \alpha< 2n,
\end{aligned}
\end{equation}
 where $\widetilde{\eta}=\frac{\sigma_{t_{\alpha}}}{\sigma_{x_{n}}}$, $\sigma_{x_{n}}(0)=1,$ $\sigma_{t_\alpha}(0)=0$.
  One has $\mathcal{T}(u-\varphi)=0$ on $\partial M\cap \bar\Omega_\delta$.
 The boundary value condition also gives for each $1\leq \alpha, \beta<n$, 
\begin{equation}\label{left-up}
(u-\varphi)_{t_{i}t_{j}}(0)= (u-\varphi)_{x_{n}}(0)\sigma_{t_{i}t_{j}}(0)
\mbox{  }\forall 1\leq i,j<2n.
\end{equation}

Let's turn our attention to the setting of complex manifolds with \textit{holomorphically flat} boundary in the sense that,
for any $p_0\in \partial M$, one can pick local holomorphic coordinates  
\begin{equation}
\begin{aligned}
\label{holomorphic-coordinate-flat}
(z_1,\cdots, z_n), \mbox{  } z_i=x_i+\sqrt{-1}y_i, 
\end{aligned}
\end{equation}
 centered at $p_0$ such that
 $\partial M$ is locally of the form $$\mathfrak{Re}(z_n)=0.$$ Furthermore we may assume $g_{i\bar j}(p_0)=\delta_{ij}$. 
Under the holomorphic coordinate \eqref{holomorphic-coordinate-flat}, we can take
\begin{equation}
\label{tangential-oper-Leviflat1}
\begin{aligned}
\mathcal{T}=D:=    \frac{\partial}{\partial x_\alpha}, \mbox{ } \frac{\partial}{\partial y_\alpha},  
  \quad 1\leq \alpha\leq n-1.
\end{aligned}
\end{equation}
 It would be worthwhile to note that 
 such local holomorphic coordinate system \eqref{holomorphic-coordinate-flat} is only needed in the proof  of Proposition \ref{mix-Leviflat}. In addition,  when $M=X\times S$,
$D= \frac{\partial}{\partial x_{\alpha}}, \mbox{ }
 \frac{\partial}{\partial y_{\alpha}},$ where  $z'=(z_1,\cdots z_{n-1})$ is local holomorphic coordinate of $X$.

For simplicity we denote the tangential operator on $\partial M$ by
\begin{equation}
\label{tangential-operator123}
\begin{aligned}
\mathcal{T}=\nabla_{\frac{\partial}{\partial t_{\alpha}}}-\gamma\widetilde{\eta}\nabla_{\frac{\partial}{\partial x_{n}}}
\end{aligned}
\end{equation}
where  
\begin{equation}
	\gamma=
	\begin{cases}
	  0 \,& \mbox{ $\partial M$ is holomorphically flat};\\
	  1 \,& \mbox{ otherwise.}\nonumber
 \end{cases}
\end{equation}

By  \eqref{asym-sigma1}  one derives $|\widetilde{\eta}|\leq C'|z|$ on $\Omega_\delta$. 
Since  $(u-\varphi)|_{\partial M}=0$  we obtain
 $\mathcal{T}(u-\varphi)|_{\partial M}=0$. Together with \eqref{key-14-yuan3}, one has
\begin{equation}\begin{aligned}\label{bdr-t}
 |(u-\varphi)_{t_{\alpha}}|\leq C|z| \mbox{ on } \partial M\cap\bar \Omega_\delta,
\mbox{  } \forall 1\leq \alpha<2n.
\end{aligned}
\end{equation}

\subsection{Completion of proof of Proposition \ref{mix-general}}
\label{Quantitative-boundes-mix}

Let's 
 denote 
\begin{equation}
\begin{aligned}
\,& F^{i\bar j}(A)=\frac{\partial F(A)}{\partial a_{i\bar j}}, \,& A=(a_{i\bar j}).  \nonumber
\end{aligned}
\end{equation}
The linearized operator of equation \eqref{mainequ} at $u$ is given by
\begin{equation}
\begin{aligned}
\mathcal{L}v = F^{i\bar j}(\mathfrak{g}[u])v_{i\bar j}.  \nonumber
\end{aligned}
\end{equation}

The following lemma plays an important role in the proof. 
 The lemma of this type goes back at least to \cite[Theorems 2.16, 2.17]{Guan12a}.
 \begin{lemma}
[{\cite[Lemma 2.2]{GSS14}}]
\label{guan2014}
 Suppose 
  \eqref{elliptic} and \eqref{concave} hold.
 Let $K$ be a compact subset of $\Gamma$ and $\beta>0$. There is a constant $\varepsilon>0$ such that,
 for  $\mu\in K$ and $\lambda\in \Gamma$, when $|\nu_{\mu}-\nu_{\lambda}|\geq \beta$,
\begin{equation}
\label{2nd}
\begin{aligned}
\sum_{i=1}^n f_{i}(\lambda)(\mu_{i}-\lambda_{i})\geq f(\mu)-f(\lambda)+\varepsilon  \left(1+\sum_{i=1}^n f_{i}(\lambda)\right).
\end{aligned}
\end{equation}
Here $\nu_{\lambda}=Df(\lambda)/|Df(\lambda)|$ denotes the unit normal vector to the level set
$\partial\Gamma^{f(\lambda)}$, where $Df(\lambda)=(f_1(\lambda),\cdots, f_n(\lambda))$.

\end{lemma}

\begin{remark}
	From the original proof of 
	\cite[Lemma 2.2]{GSS14},
	we know that the constant $\varepsilon$ in 
	\eqref{2nd}  depends only on
	$\mu$, $\beta$ and other known data.
\end{remark}

We derive quantitative boundary estimates for tangential-normal derivatives by using barrier functions. 
This type of construction of barrier functions  
goes back at least  to 
\cite{Guan1993Boundary,Guan1998The}.
We shall point out that, during the proof, 
 the constants  $C$, $C_{\Phi}$,  $C_1$, $C_1'$, $C_2$, $A_1$, $A_2$, $A_3$, etc, 
depend on  neither  $|\nabla u|$ nor  $(\delta_{\psi,f})^{-1}$. The constant $\gamma$ always stands for that from \eqref{tangential-operator123}.

By direct calculations, one derives  
 \begin{equation}
    \begin{aligned}
 \,&   u_{x_{k} l}=u_{l x_{k}}+T^{p}_{kl}u_{p}, \,&
 u_{y_{k} l}=u_{l y_{k}}+\sqrt{-1}T^{p}_{kl}u_{p},  \nonumber
 \end{aligned}
\end{equation}
 \begin{equation}
    \begin{aligned}
\,& (u_{x_k})_{\bar j}=u_{x_k\bar j}+\overline{\Gamma_{kj}^l} u_{\bar l}, 
\,& (u_{y_k})_{\bar j}=u_{y_k\bar j}-{\sqrt{-1}}\overline{\Gamma_{kj}^l} u_{\bar l}, \nonumber
    \end{aligned}
    \end{equation}
 \begin{equation}
    \begin{aligned}
    (u_{x_k})_{i\bar j}=
    u_{x_ki\bar j}+\Gamma_{ik}^lu_{l\bar j}+\overline{\Gamma_{jk}^l} u_{i\bar l}-g^{l\bar m}R_{i\bar j k\bar m}u_l,  \nonumber
     \end{aligned}
\end{equation}
 \begin{equation}
    \begin{aligned}
    (u_{y_k})_{i\bar j}=
    u_{y_ki\bar j}+\sqrt{-1}
    \left(\Gamma_{ik}^l u_{l\bar j}-\overline{\Gamma_{jk}^l} u_{i\bar l}\right)-\sqrt{-1}g^{l\bar m}R_{i\bar j k\bar m}u_l,
    \nonumber
    \end{aligned}
    \end{equation}
\begin{equation}
\label{yuan-Bd2}
\begin{aligned}
F^{i\bar j}u_{x_{k}i\bar j}
= F^{i\bar j}u_{i\bar j x_{k}}+g^{l\bar m}F^{i\bar j}R_{i\bar j k\bar m}u_{l}-2\mathfrak{Re}
\left(F^{i\bar j}T^{l}_{ik}u_{l\bar j}\right), \nonumber
 \end{aligned}
\end{equation}
 \begin{equation}
    \begin{aligned}
F^{i\bar j}u_{y_{k}i\bar j}=F^{i\bar j}u_{i\bar j y_{k}}+\sqrt{-1}g^{l\bar m}F^{i\bar j}R_{i\bar j k\bar m}u_{l}+
2\mathfrak{Im}\left(F^{i\bar j}T^{l}_{ik}u_{l\bar j}\right), \nonumber
\end{aligned}
\end{equation}
where 
\[T^{k}_{ij}= g^{k\bar l}  \left(\frac{\partial g_{j\bar l}}{\partial z_i} - \frac{\partial g_{i\bar l}}{\partial z_j}\right), \quad
 R_{i\bar jk\bar l}= -\frac{\partial^2 g_{k\bar l}}{\partial z_i \partial\bar z_j}
 + g^{p\bar q}\frac{\partial g_{k\bar q}}{\partial z_i}\frac{\partial g_{p\bar l}}{\partial \bar z_j}.\]
As a consequence, 
 \begin{equation}
 \label{linear-1}
    \begin{aligned}
 \mathcal{L}(\pm u_{t_{\alpha}}) \geq \pm\psi_{t_{\alpha}}
 -C\left(1+|\nabla u|\right)\sum_{i=1}^n f_i -C\sum_{i=1}^n f_i |\lambda_i|. 
  \end{aligned}
\end{equation}

Denote  $$b_{1}=1+\sup_{M} |\nabla u|^{2}.$$
\begin{lemma}
\label{yuan-key0}
Given $p_0\in\partial M$, let $u$ be a $C^3$  admissible solution to equation \eqref{mainequ}, and $\Phi$ is defined as
 \begin{equation}
 \label{Phi-def1}
 \begin{aligned}
 \Phi=\pm \mathcal{T}(u-\varphi)+\frac{\gamma}{\sqrt{b_1}}(u_{y_{n}}-\varphi_{y_{n}})^2 \mbox{ in } \Omega_\delta.
 \end{aligned}
 \end{equation}
 Then there is a positive constant $C_{\Phi}$ depending on
$|\varphi|_{C^{3}(\bar M)}$, $|\chi|_{C^{1}(\bar M)}$,
 $|\nabla \psi|_{C^{0}(\bar M)}$
 and other known data 
 such that  for some small positive constant $\delta$,
\begin{equation}
\label{yuan-1}
\begin{aligned}
\mathcal{L}\Phi \geq 
 -C_{\Phi}  \sqrt{b_1}  \sum_{i=1}^n f_i  - C_{\Phi}  \sum_{i=1}^n f_i|\lambda_i|
-C_\Phi  \mbox{ on } \Omega_{\delta}. \nonumber
\end{aligned}
\end{equation}

  In particular, if  $\partial M$ is holomorphically flat and $\varphi\equiv \mathrm{constant}$
then $C_{\Phi}$ depends on $|\chi|_{C^{1}(\bar M)}$,
 $|\nabla \psi|_{C^{0}(\bar M)}$
 and other known data.

\end{lemma}

\begin{proof}
Together with 
\eqref{linear-1} and Cauchy-Schwarz inequality, one obtains
\begin{equation}
\label{bdy-g1}
\begin{aligned}
\mathcal{L}(\pm\mathcal{T}u)
 \geq \,&
 -C\sqrt{b_1}\sum_{i=1}^n f_i  
  - C\sum_{i=1}^n f_i|\lambda_i|
 -\frac{\gamma}{\sqrt{b_1}} F^{i\bar j}u_{y_n i}u_{y_n \bar j}-C|\nabla\psi|, \nonumber
 \end{aligned}
\end{equation}
\begin{equation}
\begin{aligned}
F^{i\bar j}(\widetilde{\eta})_i (u_{x_{n}})_{\bar j}
\leq \,& C\sum_{i=1}^nf_i|\lambda_i|+ \frac{1}{\sqrt{b_1}}F^{i\bar j}u_{y_n i} u_{y_n \bar j} +C\sqrt{b_1}\sum_{i=1}^n f_i,
\nonumber
\end{aligned}
\end{equation}
\begin{equation}
\label{yuan-hao2}
\begin{aligned}
\mathcal{L}((u_{y_{n}}-\varphi_{y_{n}})^2)
 \geq 
F^{i\bar j}u_{y_n i}u_{y_n \bar j}-C\left(1+|\nabla u|^2\right)   \sum_{i=1}^n f_i  
  - C|\nabla u|\sum_{i=1}^n f_i|\lambda_i| 
  - C\left(1+|\nabla u|\right). \nonumber
\end{aligned}
\end{equation}
Putting these inequalities together, we complete the proof.
\end{proof}

To estimate the quantitative boundary estimates for mixed derivatives, 
we should employ barrier function of the form
\begin{equation}
\label{barrier1}
\begin{aligned}
v= (\underline{u}-u)
- t\sigma
+N\sigma^{2}   \mbox{  in  } \Omega_{\delta},
\end{aligned}
\end{equation}
where $t$, $N$ are positive constants to be determined.

 Let $\delta>0$ and $t>0$ be sufficiently small with $N\delta-t\leq 0,$ so that in $\Omega_{\delta}$,    
 \begin{equation}
\begin{aligned}
v\leq 0, \mbox{  } \sigma \mbox{ is } C^2,
 \end{aligned}
\end{equation}
\begin{equation}
\label{bdy1}
\begin{aligned}
 \frac{1}{4} \leq |\nabla \sigma|\leq 2,  \mbox{  }
  |\mathcal{L}\sigma | \leq   C_2\sum_{i=1}^n f_i. 
\end{aligned}
\end{equation}
In addition, we can choose $\delta$ and $t$ small enough such that
\begin{equation}
\label{yuanbd-11}
\begin{aligned}
|2N\delta-t|\leq \min\left\{\frac{\varepsilon}{2C_{2}}, \frac{\beta}{16\sqrt{n}C_2} \right\},
\end{aligned}
\end{equation}
where $\beta:= \frac{1}{2}\min_{\bar M} dist(\nu_{\underline{\lambda} }, \partial \Gamma_n)$,
 $\varepsilon$ is the constant corresponding to $\beta$ in Lemma \ref{guan2014}, and $C_2$ is the constant in \eqref{bdy1}.

We construct in $\Omega_\delta$ the barrier function as follows:
\begin{equation}
\label{Psi}
\begin{aligned} 
\widetilde{\Psi} =A_1 \sqrt{b_1}v -A_2 \sqrt{b_1} |z|^2 + \frac{1}{\sqrt{b_1}} \sum_{\tau<n}|\widetilde{u}_{\tau}|^2+ A_3 \Phi,
\end{aligned}
\end{equation}
where and hereafter   we denote  $$\widetilde{u}=u-\varphi.$$

Similar to \cite[Proposition 2.19]{Guan12a} one has the following lemma.
\begin{lemma} 
\label{lemma5.4-yuan}
There is an index $r$ so that 
\begin{equation}
\label{L-u-2}
\begin{aligned}\sum_{\tau<n} F^{i\bar j}\mathfrak{g}_{\bar\tau i}\mathfrak{g}_{\tau \bar j}\geq   \frac{1}{2}\sum_{i\neq r} f_{i}\lambda_{i}^{2}. \nonumber
\end{aligned}
\end{equation}
\end{lemma}

\begin{proof}
Let $U=(a_{ij})$ be a $n\times n$ unitary matrix that simultaneously diagonalizes 
$(F^{i\bar j})$ and $(\mathfrak{g}_{i\bar j})$ at a fixed point. That is
\begin{equation}
\begin{aligned}
\left(F^{i\bar j}\right)=U^*\mathrm{diag}(f_1,\cdots,f_n)U, \mbox{  } \left(\mathfrak{g}_{i\bar j}\right)=U^*\mathrm{diag}(\lambda_1,\cdots,\lambda_n)U. \nonumber
\end{aligned}
\end{equation}
Here $U^*=(b_{ij})$, $b_{ij}=\overline{a_{ji}}$. Since $U$ is unitary, $U^*=U^{-1}$.
Thus $\left(\mathfrak{g}_{i\bar j}\right)\cdot (F^{i\bar j} )\cdot\left(\mathfrak{g}_{i\bar j}\right)=U^*\mathrm{diag}(f_1\lambda_1^2,\cdots,f_n\lambda_n^2)U.$
For any fixed $1\leq\tau\leq n$, we have
\begin{equation}
\begin{aligned}
\sum_{i,j=1}^n F^{i\bar j}\mathfrak{g}_{i\bar \tau}\mathfrak{g}_{\tau\bar j}
=\sum_{i=1}^n f_i\lambda_i^2|a_{i\tau}|^2. \nonumber
\end{aligned}
\end{equation}
Consequently, 
\begin{equation}
\begin{aligned}
\sum_{\tau=1}^{n-1}\sum_{i,j=1}^n F^{i\bar j}\mathfrak{g}_{i\bar \tau}\mathfrak{g}_{\tau\bar j}
=\sum_{i=1}^n f_i\lambda_i^2\left(1-|a_{in}|^2\right)\geq \frac{1}{2} \sum_{|a_{in}|^2\leq\frac{1}{2}}f_i\lambda_i^2 \nonumber
\end{aligned}
\end{equation}
 as required. 
\end{proof}

    \begin{proof}
[Proof of Proposition \ref{mix-general}]

If $A_2\gg A_3\gg1$ then one has  $\widetilde{\Psi}\leq 0 \mbox{ on } \partial \Omega_\delta$, here we use  \eqref{bdr-t}.
Note $\widetilde{\Psi}(p_0)=0$.  
It suffices  to prove $$\mathcal{L}\widetilde{\Psi}\geq 0 \mbox{ on } \Omega_\delta,$$
which yields $\widetilde{\Psi}\leq 0 $ in $\Omega_{\delta}$, and then $(\nabla_\nu \widetilde{\Psi})(p_0)\leq 0$.

 By a direct computation one has
\begin{equation}
\label{L-v}
\begin{aligned}
\mathcal{L}v\geq F^{i\bar j} (\mathfrak{\underline{g}}_{i\bar j}-\mathfrak{g}_{i\bar j})-C_2 |2N\sigma-t|\sum_{i=1}^n f_i +2N F^{i\bar j}\sigma_i \sigma_{\bar j}. \nonumber
\end{aligned}
\end{equation}

Applying \cite[Lemma 6.2]{CNS3}, 
 with a proper permutation
of $\underline{\lambda}$ if necessary, one obtains
 \begin{equation}
\begin{aligned}
F^{i\bar j}   \mathfrak{\underline{g}}_{i\bar j}=F^{i\bar j}(\mathfrak{g})  \mathfrak{\underline{g}}_{i\bar j} \geq \sum_{i=1}^n f_i(\lambda) \underline{\lambda}_i=\sum_{i=1}^n f_i  \underline{\lambda}_i. \nonumber
 \end{aligned}
\end{equation}
Since $F^{i\bar j}   \mathfrak{g}_{i\bar j}=\sum_{i=1}^n f_i\lambda_i$, we have
\begin{equation}\begin{aligned}
F^{i\bar j} (\mathfrak{\underline{g}}_{i\bar j}-\mathfrak{g}_{i\bar j}) \geq \sum_{i=1}^n f_i(\underline{\lambda}_i-\lambda_i). \nonumber
 \end{aligned}\end{equation}

Together with Lemma \ref{lemma5.4-yuan}, some straightforward computations yield
\begin{equation}
\begin{aligned}
\mathcal{L} \left(\sum_{\tau<n}|\widetilde{u}_{\tau}|^2 \right)
\geq \,&
 \frac{1}{2}\sum_{\tau<n}F^{i\bar j} \mathfrak{g}_{\bar \tau i} \mathfrak{g}_{\tau \bar j}
   -C_1'\sqrt{b_1} \sum_{i=1}^n f_{i}|\lambda_{i}|  -C_1' b_1\sum_{i=1}^n f_{i} -C_1'\sqrt{ b_1} \\ 
\geq \,&
 \frac{1}{4}\sum_{i\neq r} f_{i}\lambda_{i}^{2}
   -C_1'\sqrt{b_1} \sum_{i=1}^n  f_{i}|\lambda_{i}|-C_1' b_1 \sum_{i=1}^n f_{i} -C_1'\sqrt{ b_1}. \nonumber
\end{aligned}
\end{equation}


According to Lemma \ref{lemma3.4} and $\sum_{i=1}^n f_i(\underline{\lambda}_i-\lambda_i) \geq 0$,
 we obtain  the following inequalities respectively
\begin{equation}\begin{aligned}
\sum_{i=1}^n f_i |\lambda_i| =2\sum_{\lambda_i\geq 0} f_i\lambda_i -\sum_{i=1}^n f_{i} \lambda_i < 2\sum_{\lambda_i\geq 0} f_i\lambda_i, \nonumber
\end{aligned}\end{equation}
 \begin{equation} \label{inequ-1}\begin{aligned}
\sum_{i=1}^n f_i |\lambda_i| = \sum_{i=1}^n f_i\lambda_i -2\sum_{\lambda_i<0} f_{i} \lambda_i \leq \sum_{i=1}^n f_i\underline{\lambda}_i-2\sum_{\lambda_i<0} f_{i} \lambda_i. \nonumber
\end{aligned}\end{equation}
In conclusion, combining with Cauchy-Schwarz inequality, we have
 \begin{equation}
 \label{flambda}
\begin{aligned}
\sum_{i=1}^n f_i |\lambda_i|
\leq  \frac{\epsilon}{4\sqrt{b_1}}\sum_{i\neq r} f_i\lambda_i^2 +\left(\sup_{\bar M}|\underline{\lambda}|+\frac{4\sqrt{b_1}}{\epsilon}\right)\sum_{i=1}^n f_i. \nonumber
\end{aligned}
\end{equation}
  
Taking $\epsilon=\frac{1}{C_1'+A_3C_\Phi }$, and
putting the above inequalities together we have
\begin{equation}
\label{bdy-main-inequality}
\begin{aligned}
\mathcal{L}\widetilde{\Psi} \geq \,&
A_1 \sqrt{b_1} \sum_{i=1}^n f_{i}(\underline{\lambda}_i-\lambda_i)
+ 2A_1N \sqrt{b_1} F^{i\bar j}\sigma_i \sigma_{\bar j}
\\ \,&
   -  \{ C_1'+ A_2+A_3C_\Phi  +A_1C_2  |2N\sigma-t|
  +4(C_1'+A_3C_\Phi)^2   
   \\
  \,&
     +(C_1'+A_3C_\Phi) \sup_{\bar M}|\underline{\lambda}| /{\sqrt{b_1}}\} \sqrt{b_1}\sum_{i=1}^n f_i
 -(C_1' +A_3 C_\Phi).
\end{aligned}
\end{equation}

Let's take $\beta= \frac{1}{2}\min_{\bar M} dist(\nu_{\underline{\lambda} }, \partial \Gamma_n)$ as above, and
let $\varepsilon$ be the positive constant in Lemma \ref{guan2014} accordingly.

{\bf Case I}: If $|\nu_{\lambda }-\nu_{\underline{\lambda} }|\geq \beta$,   then by Lemma \ref{guan2014}
 we have
\begin{equation}
\label{guan-key1}
\begin{aligned}
\sum_{i=1}^n f_{i}(\underline{\lambda}_i-\lambda_i)  \geq \varepsilon  \left(1+\sum_{i=1}^n f_i\right).
\end{aligned}
\end{equation}
Note that \eqref{yuanbd-11} implies $A_1C_2 |2N\sigma-t|\leq \frac{1}{2}A_1 \varepsilon$. Taking $A_1\gg 1$ we derive $$\mathcal{L}\widetilde{\Psi} \geq 0 \mbox{ on } \Omega_\delta.$$

{\bf Case II}: Suppose that $|\nu_{\lambda }-\nu_{\underline{\lambda} }|<\beta$ and therefore
$\nu_{\lambda }-\beta \vec{\bf 1} \in \Gamma_{n}$ and
\begin{equation}
\label{2nd-case1}
\begin{aligned}
f_{i} \geq  \frac{\beta }{\sqrt{n}} \sum_{j=1}^n f_{j}. 
\end{aligned}
\end{equation}
By \eqref{bdy1}, we have $|\nabla \sigma|\geq\frac{1}{4}$ in $\Omega_\delta$, then 
\begin{equation}
\label{bbvvv}
\begin{aligned}
A_1N \sqrt{b_1} F^{i\bar j}\sigma_i \sigma_{\bar j} \geq \frac{A_1N \beta \sqrt{b_1}}{16\sqrt{n}}\sum_{i=1}^n f_i  \mbox{ on } \Omega_\delta.
\end{aligned}
\end{equation}
This term can control 
all the bad terms containing $\sum_{i=1}^n f_i$ in \eqref{bdy-main-inequality}.
On the other hand,   $\mathcal{L}(\underline{u}-u)\geq 0$ and
the bad term  $-(C_1' +A_3 C_\Phi)$ in the last term of \eqref{bdy-main-inequality}
can be dominated by combining \eqref{bbvvv} with \eqref{sumfi1}.
Thus  $$\mathcal{L}(\widetilde{\Psi}) \geq 0 \mbox{ on }\Omega_\delta, \mbox{ if } A_1N\gg 1.$$
\end{proof}
   
In the case when $\partial M$ is 
 {holomorphically flat} and $\varphi$ is a constant, $\Phi= \pm{D}u$ in \eqref{Phi-def1} and the local barrier function  in \eqref{Psi} reads as follows
 		$$\widetilde{\Psi} =A_1 \sqrt{b_1}v -A_2 \sqrt{b_1} |z|^2 + \frac{1}{\sqrt{b_1}} \sum_{\tau<n}|{u}_{\tau}|^2\pm A_3 {D}u,$$ 
 		where $D$ is given in \eqref{tangential-oper-Leviflat1}.
So we have slightly delicate results.
\begin{proposition}
\label{mix-Leviflat}
Suppose, in addition to \eqref{elliptic}-\eqref{addistruc}, 
 \eqref{existenceofsubsolution}, \eqref{nondegenerate}  
  and $\psi\in C^1(\bar M)$, that $\partial M$ is \textit{holomorphically flat} and the boundary data $\varphi$ is of a constant.
Then 
for any $T\in  T_{\partial M}  \cap J T_{\partial M}  \mbox{ with } |T|=1$,
the admissible solution $u\in C^3(M)\cap C^2(\bar M)$ to the Dirichlet problem satisfies
\begin{equation}
\label{quanti-mix-derivative-1}
|\nabla^2 u(T,\nu)|\leq C 
\left(1+\sup_{M}|\nabla u|\right) \nonumber
\end{equation}
where $C$ depends on 
$|\psi|_{C^{1}(\bar M)}$, $|\underline{u}|_{C^{2}(\bar M)}$,
$\partial M$ 
up to second derivatives
and other known data (but neither on $(\delta_{\psi,f})^{-1}$ nor on $\sup_{M}|\nabla u|$). 
\end{proposition}

Together with Proposition \ref{proposition-quar-yuan1}, we obtain the following theorem.
 \begin{theorem}
 \label{thm2-bdy-leviflat}
 Suppose, in addition to \eqref{elliptic}-\eqref{addistruc},
 \eqref{existenceofsubsolution}, 
 \eqref{nondegenerate}  
  and $\psi\in C^1(\bar M)$, that $\partial M$ is \textit{holomorphically flat} and the boundary data $\varphi$ is of a constant.
Then for any admissible solution $u\in C^3(M)\cap C^2(\bar M)$ to Dirichlet problem \eqref{mainequ}, we have
\[\sup _{\partial M} \Delta u \leqslant C\left(1+\sup _{M}|\nabla u|^{2}\right).\] 
Here $C$ is a uniform positive constant depending only on 
 $|\psi|_{C^{1}(\bar M)}$, $|\nabla u|_{C^0(\partial M)}$, 
$|\underline{u}|_{C^{2}(\bar M)}$, $\partial M$
up to second derivatives 
and other known data (but neither on $\sup_{M}|\nabla u|$ nor on $(\delta_{\psi,f})^{-1}$).

 \end{theorem}

\section{Solving equations}
\label{solvingequation}

\subsection{Completion of the proof of Theorems \ref{thm1-1} and \ref{thm1-2}}
The following  second order estimate is essentially due to Sz\'{e}kelyhidi \cite{Gabor}.

\begin{theorem}
	\label{globalsecond}
	Suppose, in addition to \eqref{elliptic}, \eqref{concave}, \eqref{addistruc-0} and \eqref{nondegenerate},
	that there is an admissible subsolution $\underline{u}\in C^{2}(\bar M)$.
	Then for any admissible solution $u\in C^{4}(M)\cap C^{2}(\bar M)$ of
	Dirichlet problem \eqref{mainequ} with $\psi\in C^2(M)\cap C^{1,1}(\bar M)$,
	there exists a uniform positive constant  $C$ depending only on
	$|u|_{C^{0}(\bar M)}$, $|\psi|_{C^{2} (\bar M)}$, $|\underline{u}|_{C^{2}(\bar M)}$,
	$|\chi|_{C^{2}(\bar M)}$  and other known data such that
	\begin{equation}
		\label{2-se-global-Gabor}
		\begin{aligned}
			\sup_{M}|\partial\overline{\partial} u|
			\leq C \left(1+\sup_{M}|\nabla u|^{2} +
			\sup_{\partial M}|
			\partial\overline{\partial} u|\right).
		\end{aligned}
	\end{equation}
	
\end{theorem}
\begin{remark}
 	Following the outline of proof of  \cite[Proposition 13]{Gabor},	 using  Lemma \ref{guan2014} 
 	in place of \cite[Proposition 6]{Gabor},   
 	we can check that Sz\'{e}kelyhidi's second order estimate still holds for the Dirichlet problem without assuming  \eqref{addistruc-0}.
	Moreover, one can further verify that the constant $C$ in \eqref{2-se-global-Gabor} does not depend on $(\delta_{\psi, f})^{-1}$.
\end{remark}

The existence results follow from the standard continuity method and the above estimates.
We assume $\psi$, $\underline{u}\in C^{\infty}(\bar M)$.
The  general case of $\underline{u}\in C^{3}(\bar M)$ and $\psi\in C^{k,\alpha}(\bar M)$ follows by approximation  process.

Let's consider a family of Dirichlet problems as follows:
\begin{equation}
	\label{conti}
	\begin{aligned}
		F(\mathfrak{g}[u^{t}]) = (1-t)F(\mathfrak{g}[\underline{u}])+t\psi  \mbox{ in } M,  \quad
		u^{t}  =\varphi  \mbox{ on } \partial  M.
	\end{aligned}
\end{equation}
We set $$I=\left\{t\in [0,1]: \mbox{ there exists } u^{t}\in C^{4,\alpha}(\bar M) \mbox{ solving equation } \eqref{conti}\right\}.$$
Clearly $0\in I$ by taking $u^{0}=\underline{u}$.
The openness of $I$ is follows from the implicit function theorem and the estimates.

We can verify that $\underline{u}$ is the {admissible} subsolution along the whole method of continuity.
Combining Theorem \ref{thm1-bdy} with Theorem \ref{globalsecond} and Lemma \ref{lemma-c0-bc1},
we can conclude that there exists a uniform positive constant $C$
such that
\begin{equation}
	\begin{aligned}
		\sup_{M}|\Delta u^{t}| \leq C\left(1+\sup_{M} |\nabla u^{t}|^2\right). \nonumber
	\end{aligned}
\end{equation}
One thus applies the blow-up argument used  in \cite{Gabor}, extending that of \cite{Chen,Dinew2017Kolo},  to
derive the gradient estimate,
and so $|\partial\overline{\partial} u^t|$ has a uniform  bound.
Applying Evans-Krylov theorem \cite{Evans82,Krylov83},
adapted to the complex setting (cf.  \cite{TWWYEvansKrylov2015}),
and Schauder theory,  one obtains the required higher $C^{k,\alpha}$ regularities.
Consequently, $I$ is closed. 

Therefore, $I=[0,1]$. The proof is complete.


\subsection{Completion of the proof of Theorem \ref{thm2-diri-de}}
Notice that  in Theorem \ref{thm2-diri-de} we impose condition \eqref{continuity1}  rather than  \eqref{addistruc}. Therefore, in order to apply Theorems \ref{thm1-2} and \ref{thm1-bdy}, it requires to prove
\begin{lemma}
	\label{lemma1-con-addi}
	Suppose \eqref{elliptic}, \eqref{concave} and \eqref{continuity1} hold. Then  $f$ satisfies \eqref{addistruc}.
\end{lemma}

\begin{proof}
	The proof is based on Lemmas \ref{lemma3.4} and \ref{k-buchong1} below.
	By \eqref{continuity1} we have 
	\begin{equation}
		\label{t1-to-0}
		\begin{aligned}
			\lim_{t\rightarrow0^+}f(t\vec{\bf1})
			=f(\vec{\bf 0})>-\infty. \nonumber
		\end{aligned}
	\end{equation}
	Then we have \eqref{addistruc} 
	according to Lemma \ref{k-buchong1}.
	
\end{proof}

The following statement is standard and has been used in literature. 
\begin{lemma}
	\label{k-buchong1}
	Suppose $f$ satisfies \eqref{elliptic} and \eqref{concave}.
	Then for any $\sigma$
	with $\partial\Gamma^\sigma\neq \emptyset$,
	there exists $c_\sigma \vec{\bf 1}\in \partial\Gamma^\sigma$, $c_\sigma>0$.
\end{lemma}

\begin{proof}
	
As in \eqref{def-levelset}, we denote $\partial\Gamma^\sigma=\{\lambda\in\Gamma: f(\lambda)=\sigma\}$. 
	The level set $\partial\Gamma^\sigma$ is a smooth
	convex noncompact 
	hypersurface contained in $\Gamma$. Moreover, $\partial\Gamma^\sigma$ is symmetric with respect to  diagonal $\{t\vec{\bf 1}: t\in \mathbb{R}\}.$
	
	Let $\lambda^0\in \partial\Gamma^\sigma$ be the closest point to the origin. (Such a point exists, since $\partial\Gamma^\sigma$ is a closed set). The idea is to prove $\lambda^0$ is 
	the one for what we seek.
	
	Assume $\lambda^0$ is not contained in the diagonal. Then by the implicit function theorem, and the convexity and symmetry of $\partial\Gamma^\sigma$, one can choose $\lambda\in \partial\Gamma^\sigma$
so that the distance to origin is smaller than that of $\lambda^0$. It is a contradiction.
	
\end{proof}

Theorem \ref{thm2-diri-de} follows as a conclusion of 
 Theorems \ref{thm1-2} and \ref{thm1-bdy}, Lemma \ref{lemma1-con-addi} and the method of approximation.
 

 \section{Dirichlet problem on products}
  \label{sec6}
 
  \subsection{Construction of subsolutions}
  \label{sec7}
 
 According to the main theorems, the existence of subsolutions is a key ingredient to solve the Dirichlet problem.  

Let  $(M,J,\omega)=(X\times S,J,\omega)$ be a product of a closed complex manifold $X$ with a compact Riemann surface $S$ with sufficiently smooth boundary, and let $\nu$ be as above the unit inner normal vector along the boundary.
We construct strictly admissible subsolutions with 
  $\frac{\partial\underline{u}}{\partial \nu} |_{\partial M}< 0$
 on such a product. 
    It is noteworthy that $J$ is the standard induced complex structure, 
 and $\omega$ 
 is not necessary the  product metric $\omega=\pi_1^*\omega_X+\pi_2^*\omega_S$.
 
We begin with the solution $h$ to 
    \begin{equation}  
    \label{possion-def}
    \begin{aligned}
  \,& \Delta_S h =1 \mbox{ in } S, \,& h=0 \mbox{ on } \partial S.  
   \end{aligned}
  \end{equation} 
Here $\Delta_S$ is the Laplacian operator of $(S,J_S,\omega_S)$.
The existence can be found in  standard monographs, 
see e.g. \cite{GT1983}. More precisely, when $\partial S\in C^\infty$,   $h\in C^{\infty}(\bar S)$; 
while $\partial S\in C^{2,\beta}$, $0<\beta<1$,   $h\in C^\infty(S)\cap C^{2,\beta}(\bar S)$. Moreover,    $\left.\frac{\partial h}{\partial\nu}\right|_{\partial S}<0$.

We use such $h$ to construct subsolutions.
Let
    \begin{equation} 
    	\label{construct-subsolution}  \begin{aligned}
   \underline{u}=\varphi+t\pi_2^* h
  \end{aligned}\end{equation}
for  $t\gg1$
   ($\pi_2^* h=h\circ\pi_2$, still denoted by $h$  
    for simplicity),
   then
   \begin{equation}  
    \begin{aligned}
    \mathfrak{g}[\underline{u}]=\mathfrak{g}[\varphi]+t\pi_2^*\omega_S, \mbox{  } \mbox{ }
    \left.\frac{\partial \underline{u}}{\partial \nu}\right|_{\partial M}<0 \mbox{ for } t\gg1. \nonumber
    \end{aligned}
  \end{equation}
   Therefore,  $\underline{u}$ is the subsolution if condition \eqref{cone-condition1} holds.
 Furthermore, since the boundary data $\varphi$ is admissible,   \eqref{cone-condition1} always holds when 
 imposed \eqref{unbounded}.

Significantly, according to 
Lemma \ref{yuan's-quantitative-lemma}, if $\omega=\pi_1^*\omega_X+\pi_2^*\omega_S$ and $\chi$ splits by $\chi=\pi_1^*\chi_1+ \pi_2^*\chi_2,$ where 
$\chi_1$ is a real $(1,1)$-form on $X$, 
$\chi_2$ is a real $(1,1)$-form on $S$, 
 then condition \eqref{cone-condition1} reduces to 
\begin{equation}
	\label{cone-condition1-1}
	\begin{aligned}
		\lim_{t\rightarrow +\infty} f(\lambda'(\chi_1), t)>\psi 
		\mbox{ and  } \lambda_{\omega_X}(\chi_1)\in \Gamma_\infty  \mbox{ in } \bar M,   \nonumber
	\end{aligned}
\end{equation}
where $\lambda'(\chi_1)$ are the eigenvalues of $\chi_1$ with respect to $\omega_X$.
This indicates that, 
for $\chi=\pi_1^*\chi_1+ \pi_2^*\chi_2,$ the solvability of Dirichlet problem is heavily determined by 
$\chi_1$ rather than by $\chi_2$.

\subsection{The Dirichlet problem with less regular boundary} 
\label{sec8}

 A somewhat remarkable fact to us is that the regularity assumptions on boundary can be further weakened.
    The motivation is primarily based on Theorem \ref{thm2-bdy-leviflat} which state that,
  when $\partial M$ is {holomorphically flat} and the boundary value is  a constant, 
  the quantitative boundary estimate \eqref{bdy-sec-estimate-quar1}  
  depends only on $\partial M$ up to second derivatives and other known data. 
  Besides,  a result due to Silvestre-Sirakov \cite{Silvestre2014Sirakov} allows one to derive
 $C^{2,\alpha}$ boundary regularity with only assuming $C^{2,\beta}$ boundary.
Together with the construction of subsolution, 
we can study the Dirichlet problem on the products
with less regular boundary.\renewcommand{\thefootnote}{\fnsymbol{footnote}}\footnote{We emphasize that the geometric quantities of $(M,\omega)$ (curvature $R_{i\bar j k\bar l}$ and the torsion $T^k_{ij}$) keep bounded as approximating to $\partial M$, and all derivatives of ${\chi}_{i\bar j}$ has continues extensions to $\bar M$,
 whenever $M$ has less regularity boundary. Typical examples are as follows: $M\subset M'$, 
 $\mathrm{dim}_{\mathbb{C}}M'=n$, 
$\omega=\omega_{M'}|_{M}$ and the given data ${\chi}$ can be smoothly defined on $M'$.}

   \begin{theorem}
  \label{mainthm-09}
   Let  $(M, J,\omega)=(X\times S,J, \omega)$ be as above with
    $\partial S\in C^{2,\beta}$, $0<\beta<1$.
   Suppose in addition that  \eqref{elliptic}, \eqref{concave} and \eqref{cone-condition1} hold. 
   Then we have two conclusions:
   \begin{itemize}
   \item   Equation \eqref{mainequ} has a 
   unique $C^{2,\alpha}$ admissible solution with $u|_{\partial M}=0$
   for some 
   $0<\alpha\leq\beta$, provided that 
 $\psi\in C^2(\bar M)$, $\inf_{M} \psi> \sup_{\partial \Gamma}f$  and $f$ satisfies \eqref{addistruc}.
   \item Suppose in addition that
     $f\in C^\infty(\Gamma)\cap C(\bar\Gamma)$, $\psi\in C^{1,1}(\bar M)$ and   
    $\inf_{M} \psi=sup_{\partial \Gamma}f$. Then the Dirichlet problem  \eqref{mainequ}   has a weak solution   with $u|_{\partial M}=0$,
   $u\in C^{1,\alpha}(\bar M)$, $\forall 0<\alpha<1$  $\lambda(\mathfrak{g}[u])\in \bar \Gamma$ and $\Delta u \in L^{\infty}(\bar M)$.
   \end{itemize}
   \end{theorem}

 \begin{proof}
It only requires to consider the nondegenerate case: For some $\delta_0>0$,
 \begin{equation}
 \begin{aligned}
 \psi>\sup_{\partial\Gamma}f+\delta_0 \mbox{ in } M.
 \end{aligned}
 \end{equation}
  
The first step is to construct  approximate Dirichlet problems with constant boundary value data.
Let $h$ be the solution to \eqref{possion-def}. For $t\gg1$, $\underline{u}=th$ satisfies
 \begin{equation}
 \begin{aligned}
 f(\lambda(\mathfrak{g}[\underline{u}]))\geq \psi+\delta_1 \mbox{ in } M 
 \end{aligned}
 \end{equation}
  for some $\delta_1>0.$
Note that $$h\in C^\infty(S)\cap C^{2,\beta}(\bar S) \mbox{ and } \left.\frac{\partial h}{\partial\nu}\right|_{\partial S}<0,$$
 we get a sequence of level sets of $h$, say $\{h=-\alpha_k\}$
 and  a family of smooth Riemann surfaces $S_{k}$ enclosed by  $\{h=-\alpha_k\}$,
  such that 
 $\cup S_{k}=S$ and $\partial S_{k}$ converge to $\partial S$ in the norm of $C^{2,\beta}$. 
  Denote $M_{k}=X\times S_{k}$. For any $k\geq 1$, there exists a $\psi^{(k)}\in C^\infty(\bar M_k)$ 
 such that
  \begin{equation}
 \begin{aligned}
 |\psi-\psi^{(k)}|_{C^2(\bar M_k)}\leq 1/k. 
 \end{aligned}
 \end{equation}
 For $k\gg1$ 
  we have
 \begin{equation}
\label{App-Diri1-subsolution}
\begin{aligned}
\,& 
f(\lambda(\mathfrak{g}[\underline{u}]))\geq \psi^{(k)} +{\delta_1}/{2} \mbox{ in } M_{k}, \,& 
\underline{u}=-t\alpha_k \mbox{ on }   \partial M_{k}
\end{aligned}\end{equation} 
  which is a strictly admissible subsolution to approximate Dirichlet problem
\begin{equation}
\label{App-Diri1}
\begin{aligned}
\,& 
f(\lambda(\mathfrak{g}[u]))= \psi^{(k)} \mbox{ in } M_{k}, \,& 
u =-t\alpha_k \mbox{ on }   \partial M_{k}.
\end{aligned}\end{equation} 
 According to 
  Theorems \ref{thm1-1} and \ref{thm1-2},  the Dirichlet problem \eqref{App-Diri1} admits a unique  smooth  admissible solution $u^{(k)}\in C^{\infty}(\bar M_k)$. 
  Moreover, notice the boundary data of \eqref{App-Diri1} is a constant, Theorem \ref{thm2-bdy-leviflat}applies. Therefore,
  \begin{equation}
 \label{uniform-00}
\begin{aligned}
\sup_{M_k}\Delta u^{(k)}\leq C_k \left(1+\sup_{M_{k}}|\nabla u^{(k)}|^2\right)
\end{aligned}
\end{equation}
where $C_k$ is a constant depending on $|u^{(k)}|_{C^0(M_{k})}$, $|\nabla u^{(k)}|_{C^0(\partial M_{k})}$, $|\underline{u}|_{C^{2}(M_{k})}$,   $|\psi^{(k)}|_{C^{2}(M_{k})}$,
$\partial M_{k}$ up to second order derivatives  and other known data (but not on $(\delta_{\psi^{(k)},f})^{-1}$).

It requires to prove that 
  \begin{equation}
 \label{uniform-c0-c1}
\begin{aligned}
|u^{(k)}|_{C^0(M_{k})}+\sup_{\partial M_{k}}|\nabla u^{(k)}|\leq C, \mbox{ independent of } k.
\end{aligned}
\end{equation}
Let $w^{(k)}$ be the solution of
 \begin{equation}
 \label{supersolution-k}
\begin{aligned}
\,& \Delta w^{(k)} +\mathrm{tr}_\omega{\chi} =0 \mbox{ in } M_{k}, \,& w^{(k)}= -t\alpha_k \mbox{ on } \partial M_{k}. \nonumber
\end{aligned}
\end{equation}
By the maximum principle and the boundary value condition, 
one has
\begin{equation}
\label{approxi-boundary1}
\begin{aligned}
\,& \underline{u} \leq u^{(k)}\leq w^{(k)} \mbox{ in } M_{k}, \,&
 \frac{\partial \underline{u}}{\partial \nu} \leq  \frac{\partial u^{(k)}}{\partial \nu} \leq \frac{\partial w^{(k)}}{\partial \nu} 
   \mbox{ on } \partial M_{k}.  \nonumber
\end{aligned}
\end{equation}
 On the other hand, for $t\gg1$, 
we have  
\begin{equation}
 \begin{aligned}
  \Delta (-\underline{u}-2t\alpha_k)+\mathrm{tr}_\omega \chi=\,& -t\mathrm{tr}_\omega(\pi_2^*\omega_S)+\mathrm{tr}_\omega\chi\leq 0 \,& \mbox{ in } M_k,\\
   -\underline{u}-2t\alpha_k=\,& -t\alpha_k \,& \mbox{ on } \partial M_k.  \nonumber
 \end{aligned}
 \end{equation}
As a result, we have
 \begin{equation}
 \begin{aligned}
\,& w^{(k)}  \leq -\underline{u} -2t\alpha_k \mbox{ in } M_{k},  \,&  w^{(k)}  =-\underline{u}-2t\alpha_k   \mbox{ on } \partial M_{k}, \nonumber
 \end{aligned}
 \end{equation}
which further implies
\[\frac{\partial w^{(k)}}{\partial \nu} \leq -\frac{\partial \underline{u}}{\partial \nu} \mbox{ on } \partial M_k\] 
as required. 
 Consequently, 
 \eqref{uniform-00} holds for a uniform constant $C'$ which does not depend on $k$. 
Thus $$|u|_{C^2(M_{k})}\leq C, \mbox{ independent of } k.$$ 

To complete the proof, we apply Silvestre-Sirakov's \cite{Silvestre2014Sirakov} result to 
derive $C^{2,\alpha'}$ estimates on the boundary, while
the convergence of  $\partial M_{k}$ in the norm $C^{2,\beta}$ 
 allows  us to take a limit ($\alpha'$ can be uniformly chosen).

\end{proof}


For the Dirichlet problem on $M=X\times S$ with homogeneous boundary value,
according to Theorem \ref{mainthm-09},
it is only requires to assume   $\partial M\in C^{2,\beta}$. 
Such a regularity assumption on the boundary  
is impossible for homogeneous Monge-Amp\`ere equation on certain
bounded domains in $\mathbb{R}^n$ as shown by
 Caffarelli-Nirenberg-Spruck \cite{CNS-deg},
where the $C^{3,1}$-regularity assumptions on boundary and boundary value are optimal for the $C^{1,1}$ global regularity of the weak 
solution to homogeneous real Monge-Amp\`ere equation on $\Omega\subsetneq\mathbb{R}^n$.
 Additionally, this is also different from the case for Dirichlet problem of
  nondegenerate real Monge-Amp\`ere equation on certain bounded domains $\Omega\subsetneq \mathbb{R}^2$,
as shown by Wang \cite{WangXujia-1996} the optimal regularity assumptions on
 boundary and boundary value  are both $C^3$-smooth.
  We refer to   \cite{Figalli2017MA,Guedj2017Zeriahi} and references therein
  for more results regarding to Monge-Amp\`ere equation with less regular right-hand side.

\section{Other results}
\label{sec9}

\subsection{Uniqueness of weak solution} \label{uniqueness-weak-solution}
Following Chen \cite{Chen}, we define 
\begin{definition}
\label{def-c0-weak}
A continuous function $u\in C(\bar M)$ is   a weak $C^0$-solution to the degenerate equation \eqref{mainequ}  with prescribed boundary data $\varphi$, if
for any $\epsilon>0$ there is a $C^2$-{admissible} function $\widetilde{u}$ such that $$\left|u-\widetilde{u}\right|<\epsilon,$$ 
 where $\widetilde{u}$ solves
\begin{equation}
\begin{aligned}
F(\mathfrak{g}[\widetilde{u}])=\psi+\rho_{\epsilon} \mbox{ in } M, \mbox{  } \widetilde{u}=\varphi \mbox{ on } \partial M. \nonumber
\end{aligned}
\end{equation}
Here $\rho_{\epsilon}$ is a function satisfying $0<\rho_{\epsilon}<C(\epsilon)$, and $C(\epsilon)\rightarrow 0$ as $\epsilon\rightarrow 0$.
\end{definition}

\begin{theorem}
\label{weakc0comparison}
Suppose $u^1$, $u^2$ are two $C^0$-weak solutions to the degenerate equation  \eqref{mainequ} with 
boundary data $\varphi^1$, $\varphi^2$. 
Then $$\sup_{M}|u^1-u^2|\leq \sup_{\partial M}|\varphi^1-\varphi^2|.$$
\end{theorem}

The proof is almost parallel to that of   \cite[Theorem 4]{Chen}. 
We  omit it here. 
\begin{corollary}
\label{unique-weak-solution}
The weak $C^0$-solution to the Dirichlet problem \eqref{mainequ} for degenerate equation 
 is unique, provided the boundary data is fixed.
\end{corollary}

\subsection{Construction of subsolutions 
revisited}

Let $\Omega$ be a bounded strictly pseudoconvex domain in $\mathbb{C}^m$, $2\leq m\leq n-1$, with smooth boundary 
$\partial \Omega$, let $(X,J_X,\omega_X)$ be a closed Hermitian manifold of complex dimension $n-m$.
  For the Dirichlet problem \eqref{mainequ} satisfying 
  \begin{equation}\label{unbounded-buchong-m}\begin{aligned}
  		\lim _{t \rightarrow+\infty} f(\lambda_{1}, \cdots, \lambda_{n-m},\lambda_{n-m+1}+t,  \cdots \lambda_{n}+t)=\sup_{\Gamma} f, \mbox{ } \forall \lambda=\left(\lambda_{1}, \cdots, \lambda_{n}\right) \in \Gamma,  
  	\end{aligned}
  \end{equation}
  we can construct strictly {admissible} subsolutions on 
  $(M,J,\omega)=(X\times \Omega,J,\omega)$ (not necessary to be the standard one $\omega=\pi_1^*\omega_X+\pi_2^*\omega_\Omega$).
 
Since $\Omega$ is a smooth  bounded strictly pseudoconvex domain, there exists a smooth strictly plurisubharmonic function 
$w$ with
\begin{equation} \begin{aligned}
\,& \sqrt{-1} \partial\overline{\partial}w \geq \omega_\Omega \mbox{ in }  \Omega, \,& w=0 \mbox{ on } \partial \Omega. \nonumber
 \end{aligned} \end{equation}
Given an {admissible} boundary value $\varphi$, the subsolution is given by 
  \begin{equation}
 \begin{aligned}
 \underline{u}=tw+\varphi \mbox{ for } t\gg1. \nonumber
 \end{aligned}
 \end{equation}

As a consequence, we obtain the following theorem.
\begin{theorem}
Let $(M,J,\omega)=(X\times\Omega,J,\omega)$ be a product as above. Assume that
 $f$ satisfies \eqref{elliptic}, \eqref{concave}, \eqref{addistruc} and \eqref{unbounded-buchong-m}.
For   $\varphi\in C^\infty(\partial M)$, $\psi\in C^\infty(\bar M)$ satisfying $\inf_M\psi>\sup_{\partial \Gamma}f$, the Dirichlet problem \eqref{mainequ} is uniquely solvable in class of smooth admissible functions.
 
\end{theorem}


\begin{thebibliography}{99}
	
	
	
	
	
	
	
	
	
	\bibitem{Blocki09gradient}
	Z. B{\l}ocki,
	\textit{A gradient estimate in the Calabi-Yau theorem},
	{Math. Ann.} {\bf344} (2009),  317--327.
	
	\bibitem{Blocki09geodesic}
	Z. B{\l}ocki,
	\textit{On geodesics in the space of K\"ahler metrics}, Advances in geometric analysis, 319, Adv. Lect. Math. (ALM), 21, Int. Press, Somerville, MA, 2012.
	
	
	\bibitem{Boucksom2012}
	S. Boucksom,
	\emph{Monge-Amp\`{e}re equations on complex manifolds with boundary},
	Complex Monge-Amp\`{e}re equations and geodesics in the space of K\"{a}hler metrics.
	Lecture Notes in Math.  2038, Springer, Heidelberg, 2012,  257--282.
	
	
	
	
	
	
	
	\bibitem{CKNS2}
	L. Caffarelli, J. Kohn, L. Nirenberg and J. Spruck,
	\textit{The Dirichlet problem for nonlinear second-order elliptic equations,
		II. Complex Monge-Amp\`{e}re, and uniformly elliptic, equations},
	{Comm. Pure Appl. Math.}  {\bf 38} (1985), 209--252.
	
	
	\bibitem{CNS3}
	L.   Caffarelli, L. Nirenberg and J. Spruck, \emph{The Dirichlet
		problem for nonlinear second-order elliptic equations
		III: Functions of eigenvalues of the Hessians},
	{Acta Math.} {\bf 155} (1985),  261--301.
	
	\bibitem{CNS-deg}
	L.  Caffarelli, L. Nirenberg and J. Spruck,
	\textit{The Dirichlet problem for the degenerate Monge-Amp\`{e}re equation},
	{Rev. Mat. Iberoamericana}  {\bf 2} (1986), 19--27.
	
	
	\bibitem{Chen}
	X.-X. Chen, \textit{The space of K\"{a}hler metrics}, {J. Differential Geom.} {\bf 56} (2000),  189--234.
	
	
	
	
	\bibitem{Chen2008Tian} X.-X. Chen  and  G. Tian, \textit{Geometry of K\"ahler metrics and foliations by holomorphic discs}, Publ. Math. Inst. Hautes \'Etudes Sci. {\bf107} (2008), 1--107.
	
	
	\bibitem{Collins2019Picard} T. Collins and S. Picard, \textit{The Dirichlet problem for the $k$-Hessian equation on a complex manifold},  arXiv:1909.00447.
	
	
	
	\bibitem{Dinew2017Kolo}
	S. Dinew and S. Ko{\l}odziej, \textit{Liouville and Calabi-Yau type theorems for complex Hessian equations}, Amer. J. Math. {\bf 139} (2017),  403--415.
	
	
	\bibitem{Donaldson99} S. K. Donaldson,
	\textit{Symmeric spaces, K\"{a}hler geometry and Hamiltonian dynamics, Northern California Symplectic geometry seminar},
	{American Mathematical Society Translations, Series 2, 196. American Mathematical Society}. Providence (1999).
	
	\bibitem{Donaldson2002Holomorphic} {S. K. Donaldson}, \textit{Holomorphic discs and the complex Monge-Amp\`ere equation}, J. Symplectic Geom. {\bf 1} (2002), 171--196.
	
	
	\bibitem{Evans82}
	L. C. Evans,
	\emph{Classical solutions of fully nonlinear convex, second order elliptic equations},
	{Comm. Pure Appl. Math.}  {\bf 35} (1982), 333--363.
	
	
	\bibitem{Figalli2017MA}
	A. Figalli, {The Monge-Amp\`ere equation and its applications}.
	Zurich Lectures in Advanced Mathematics. European Mathematical Society (EMS), Z\"urich, 2017.
	
	
	
	
	
	
	
 \bibitem{GT1983} D. Gilbarg and N. Trudinger,  {Elliptic partial differential equations of second order}.  Classics in Mathematics, Springer-Verlag, Berlin, reprint  of the 1998 Edition, 2001.
	
	
	\bibitem{Guan1998The}
	B. Guan, \textit{The Dirichlet problem for complex Monge-Amp\`{e}re equations and regularity of the pluri-complex Green function}, Comm. Anal. Geom.  {\bf6} (1998), 687--703.
	
	
	\bibitem{Guan12a} B. Guan, \emph{Second order estimates and regularity for fully nonlinear elliptic equations on Riemannian manifolds}, Duke Math. J. {\bf 163} (2014), 1491--1524.
	
	
	
	\bibitem{Guan2010Li}
	B. Guan and Q. Li,
	\textit{Complex   Monge-Amp\`{e}re equations and totally real submanifolds},
	{Adv. Math.}  {\bf225} (2010), 1185--1223.
	
	
	\bibitem{GSS14} 
	B. Guan,  S.-J. Shi and  Z.-N. Sui,
	\textit{On estimates for fully nonlinear parabolic equations on Riemannian manifolds},
	{Anal. PDE.}  {\bf8} (2015), 1145--1164.
	
	
	\bibitem{Guan1993Boundary}
	B. Guan and J. Spruck,
	\textit{Boundary-value problems on $\mathbb{S}^{n}$ for surfaces of constant
		Gauss curvature},
	{Ann. of Math.} {\bf 138} (1993), 601--624.
	
	
	
	
	\bibitem{Guan2015Sun}
	B. Guan and W. Sun,
	\textit{On a class of fully nonlinear elliptic equations on Hermitian manifolds},
	{Calc. Var. PDE.}  {\bf 54}  (2015), 901--916.
	
	
	
	\bibitem{GuanP2002The} P.-F. Guan, \textit{The extremal function associated to intrinsic norms}, {Ann. of Math.}  {\bf 156}  (2002),  197--211.
	
	\bibitem{Guan2009Zhang}
	P.-F. Guan and X. Zhang,
	\textit{Regularity of the geodesic equation in the space of Sasaki metrics},
	{Adv. Math.}  {\bf230} (2012), 321--371.
	
	\bibitem{Guedj2017Zeriahi}
	V. Guedj and A. Zeriahi,
	{Degenerate complex Monge-Amp\`ere equations}.
	EMS Tracts in Mathematics, 26. European Mathematical Society (EMS), Z\"urich, 2017.
	
	

	\bibitem{Hanani1996}
	A. Hanani,
	\emph{Equations du type de Monge-Amp\`{e}re sur les varietes   hermitiennes compactes},
	{J. Funct. Anal.}, {\bf 137} (1996), 49--75.
	
	\bibitem{Hoffman1992Boundary}
	D. Hoffman, H. Rosenberg and J. Spruck,
	\emph{Boundary value problems for surfaces of constant Gauss curvature},
	{Comm. Pure Appl. Math.}  {\bf 45} (1992),  1051--1062.
	
	\bibitem{HouMaWu2010}
	Z.-L. Hou, X.-N. Ma and D.-M. Wu,
	\textit{A second order estimate for complex Hessian equations on a compact K\"{a}hler manifold},
	Math. Res. Lett. {\bf17} (2010),  547--561.
	
	\bibitem{Ivochkina1981}
	N. Ivochkina,
	\emph{The integral method of barrier functions and the Dirichlet problem for equations with operators of the Monge-Amp\`{e}re type},
	(Russian) Mat. Sb. (N.S.), {\bf 112} (1980), 193--206;
	English transl.: Math. USSR Sb., {\bf 40} (1981), 179--192.
	
	
	
	
	
	
	
	
	
	\bibitem{Krylov83}
	N. V. Krylov,
	\emph{Boundedly nonhomogeneous elliptic and parabolic equations in a domain}, 
	{Izvestia Math. Ser.}  {\bf 47} (1983), 75--108.
	English transl.: Math. USSR Izvestija {\bf 22} (1084), 67--97.
	
	
	
	
	
	\bibitem{LiSY2004} S.-Y. Li, \emph{On the Dirichlet problems for symmetric function equations of the eigenvalues of the complex Hessian}, Asian J. Math.  {\bf 8} (2004),  87--106.
	
	\bibitem{LiYY1990} Y.-Y. Li,  {\em Some existence results of fully nonlinear elliptic equations of Monge-Amp\`{e}re type}, Comm. Pure Appl. Math.  {\bf43}  (1990), 233--271.
	
	
	\bibitem{Mabuchi87}
	T. Mabuchi, \textit{Some symplectic geometry on K\"{a}hler manifolds, $I$}, {Osaka J. Math.} {\bf 24}  (1987), 227--252.
	
	 \bibitem{Marcus1956} M. Marcus, {\em An eigenvalue inequality for product of normal matrices}, Amer. Math, Monthly. {\bf 63} (1956), 173--174.
	
	
	
	
	
	

	
	\bibitem{Phong-Sturm2010}  D. H. Phong and J. Sturm,
	\textit{The Dirichlet problem for degenerate complex Monge-Ampere equations}, Comm. Anal. Geom. {\bf 18} (2010),   145--170.
	
	
	
	
	
	
	
	\bibitem{Semmes92} S. Semmes,
	\textit{Complex Monge-Amp\`{e}re and symplectic manifolds}, {Amer. J. Math.} {\bf114}  (1992), 495--550.
	
	\bibitem{Silvestre2014Sirakov}
	L. Silvestre and B. Sirakov, \textit{Boundary regularity for viscosity solutions of fully nonlinear elliptic equations},
	Comm. Partial Differential Equations {\bf 39} (2014), 1694--1717.
	
	
	\bibitem{Gabor} G. Sz\'{e}kelyhidi,  \emph{Fully non-linear elliptic equations on compact Hermitian manifolds}, 
	J. Differential Geom. {\bf 109} (2018),  337--378.
	
	\bibitem{STW17} G. Sz\'ekelyhidi, V. Tosatti and B. Weinkove, {\em Gauduchon metrics with prescribed volume form}, Acta Math. {\bf 219} (2017),  181--211.
	
	
	\bibitem{TW17} V. Tosatti and B. Weinkove, {\em The Monge-Amp\`ere equation for $(n-1)$-plurisubharmonic functions on a compact K\"ahler manifold}, J. Amer. Math. Soc. {\bf 30} (2017),   311--346.
	
 \bibitem{TW19} V. Tosatti and B. Weinkove, {\em Hermitian metrics, $(n-1,n-1)$ forms and Monge-Amp\`ere equations}, J. Reine Angew. Math. {\bf 755} (2019), 67--101.
	
	\bibitem{TWWYEvansKrylov2015}
	V. Tosatti, Y. Wang,   B. Weinkove and X.-K. Yang,
	\emph{$C^{2,\alpha}$ estimate for nonlinear elliptic equations
		in complex and almost complex geometry}, Calc. Var. PDE.
	{\bf 54} (2015), 431--453.
	
	
	
	\bibitem{Trudinger95} N. Trudinger, \emph{On the Dirichlet problem for Hessian equations},  {Acta Math.}  {\bf 175} (1995), 151--164.
	
	\bibitem{Urbas2002} J. Urbas,  \emph{Hessian equations on compact Riemannian manifolds},  Nonlinear problems in mathematical physics and related topics, II, vol. 2 of Int. Math. Ser. (N. Y.), Kluwer/Plenum, New York, 2002, pp. 367--377. 
	
	
	
	\bibitem{WangXujia-1996}
	{X.-J. Wang},
	\textit{Regularity for Monge-Amp\`ere equation near the boundary},
	Analysis {\bf 16} (1996),   101--107.
	
	
	\bibitem{Yau78}
	{S.-T. Yau},
	\textit{On the Ricci curvature of a compact K\"{a}hler manifold and the complex Monge-Amp\`{e}re equation,  $I$},
	{Comm. Pure Appl. Math.} {\bf 31} (1978), 339--411.
	
	
	
	
	\bibitem{yuan2021cvpde} R.-R. Yuan, \textit{On the Dirichlet problem for a class of fully nonlinear elliptic equations}, 
	Calc. Var. PDE. {\bf60} (2021), no. 5, Paper No. 162, 20 pp. 
	
	
	
	
	
	\bibitem{ZhangDk}
	{D.-K. Zhang},
	\textit{Hessian equations on closed Hermitian manifolds}, Pacific J. Math. {\bf291} (2017),  485--510.
	
	
	
	
	
	
\end{thebibliography}
\end{document}